\documentclass[a4paper,11pt]{article}
\usepackage{amsfonts}
\usepackage{amsmath}
\usepackage{amsthm}
\usepackage{fullpage,relsize,cite,graphicx,mathrsfs,comment,caption,subcaption,enumitem}
\usepackage{color}
\newtheorem{lemma}{Lemma}
\newtheorem{corollary}{Corollary}
\newtheorem{definition}{Definition}
\newtheorem{theorem}{Theorem}
\newtheorem{problem}{Problem}

\newtheorem{proposition}{Proposition}
\newtheorem{remark}{Remark}

\newcommand{\DeltaG}{\Delta_\Gamma}
\newcommand{\nablaG}{\nabla_\Gamma}
\newcommand{\mat}{\dot{\partial}}
\newcommand{\A}{\mathcal{H}}
\newcommand{\Je}{\mathcal{J}_\varepsilon}
\newcommand{\unit}{\texttt{1}\!\!\texttt{l}}

\def \D #1{\underline{D}_{#1}}

\numberwithin{equation}{section}
\numberwithin{lemma}{section}
\numberwithin{corollary}{section}
\numberwithin{theorem}{section}
\numberwithin{proposition}{section}
\numberwithin{problem}{section}
\numberwithin{remark}{section}
\numberwithin{definition}{section}
\title{Small deformations of Helfrich energy minimising surfaces with applications to biomembranes}

\author{Charles M. Elliott, Hans Fritz, Graham Hobbs}

\date{}

\begin{document}
\setlength\parindent{0pt}
\maketitle

\abstract{In this paper we introduce a mathematical model for small deformations  
induced by  external forces  of closed surfaces 
that are minimisers of Helfrich-type energies. Our model is suitable for the
study of deformations of cell membranes induced by the cytoskeleton. 
We describe the deformation of the surface as a graph over the undeformed surface.
A new Lagrangian and the associated Euler-Lagrange equations 
for the height function of the graph are derived. This is the natural generalisation of the well known linearisation in the 
Monge gauge for initially flat surfaces. We discuss energy perturbations of  point constraints and point forces 
acting on the surface. We establish existence and uniqueness results for weak solutions on spheres and on tori.
Algorithms for the computation of numerical solutions in the general setting are provided. 
We present numerical examples which highlight the behaviour of the surface deformations in different settings
at the end of the paper.}

\\

\\

%%%%%%%%%%%%%%%%%% SECTION %%%%%%%%%%%%%

\section{Introduction}
\subsubsection*{Motivation}
In this paper, we study small deformations of closed surfaces described by Helfrich-type energies. The surface deformations are assumed to be induced by external forces
that lead to small displacements.
The deformed surface will be described as a graph over the surface that minimizes the associated energy without any external forces.
We will simplify the problem to determine the shape of the deformed surface by introducing a quadratic approximation of the original energy.
This will lead to solving a linear elliptic PDE of fourth order for the height function instead of solving a non-linear PDE. 
We will then consider point forces and point constraints acting on the surface. 
In particular, we will show the existence and uniqueness of solutions to the associated Euler-Lagrange equations in each case.
The main purpose of the paper is to establish a rigorously derived, simplified model
for the description of deformations of surfaces, which can be applied in different settings. In particular we obtain well posed PDEs on  known simple surfaces viz. the sphere and Clifford torus. This will partially extend the results of \cite{EllGraHobKorWol15}, which studies similar deformations in the Monge gauge of an initially flat surface, to some specific closed surfaces. Furthermore the approach we take in formulating the simplified model could be generalised to other energy functionals and appropriate surfaces.   

Before we introduce the energy functional,
which we aim to study in this context, we will first give a short explanation of why our approach might be interesting in the modelling of cell morphology.  
Biomembranes are thin bilayers which surround cells and some organelles creating a barrier between their interior and exterior. 
They are principally composed of a lipid bilayer and a  large number of protein molecules.
In Biophysics, the mechanical properties of lipid bilayers are usually described in terms of the Helfrich energy (see below). 
Membrane proteins have a great variety of forms and purposes. They may be attached to a surface of the bilayer, embedded into it or span the entire bilayer. 
The proteins interact with the membrane and cause local deformations. 
Due to the great variety of proteins and their possible interactions with the membrane there are many models for membrane deformation; 
see, for example, in \cite{McMGal05} for a discussion of the main mechanisms of membrane deformation. 
A large area of study is the effect of embedded inclusions which deform the membrane by their shape, see \cite{DomFouGal98,KozSch15}.
The cross section of membrane proteins and the membrane surface usually differ dramatically in scale. 
In this paper, we will therefore totally neglect the finite size of membrane proteins, although our model could, in principle, be extended into the direction of finite size particles. 
This means that we will here only consider effects that can be modelled by point forces and point constraints; for more general force distributions see \cite{EvaTurSen03}.     

Another important mechanism for membrane deformations involves the cytoskeleton of the cell.
The cytoskeleton is a dynamic intracellular structure which defines the cell's shape.
Actin filaments \cite{GovGop06,VekGov07} are one part of the cytoskeleton.
Their polymerisation can induce forces on the cell membrane leading to filopodia formation \cite{AtiWirSun06,IsaManKacYocGov13, OrlNaoGov14}. 
In fact, the actin polymerisation is a dynamic process involving the movement of activators through the membrane; see \cite{GovGop06,VekGov07}. 
In this paper, we consider the specific case when point-like forces are applied to a surface. 
We think that this is indeed a good model for the interaction of the cytoskeleton with the membrane.
We will also introduce point constraints for the surface position to model filaments which are bound to the membrane.
An interesting question is to study the membrane mediated interactions between filaments
and to determine the optimal placement of the forces with respect to an appropriate membrane energy. 
Such interactions were not included in the model used in \cite{GovGop06}.  
Their study can be regarded as a first application of the presented approach.
If one is interested in long range interactions, the Monge gauge setting (see below) is not suitable any more.
Instead, a global analysis on closed surfaces becomes necessary.
Our approach is therefore of particular interest for the study of global membrane behaviour.

\subsubsection*{Mathematical model}           
To model the mechanical properties of the membrane its thickness is negated and it is regarded as a single elastic sheet which occupies some two dimensional hypersurface 
$\Gamma \subset \mathbb{R}^3$. We measure its elastic energy by the Helfrich energy \cite{Hel73}, accounting for possible surface tension, that is
\begin{equation}
E_{\textnormal{Helfrich}}(\Gamma) := \int_\Gamma \frac{1}{2}\kappa (H-c_0)^2 +\kappa_G K+\sigma \;do.
 \label{H_energy}
\end{equation}
Here, $H$ denotes the mean curvature and $K$ the Gaussian curvature. The constants $\kappa,\kappa_G >0$ and $\sigma \geq 0$ measure the bending rigidity and the surface tension of the membrane,
respectively. The remaining parameter $c_0$ is the spontaneous curvature, that is the preferred curvature of the membrane. 
This may be due to asymmetry in the configuration of the lipid bilayer. We will here assume that $c_0=0$,
although a generalization to non-vanishing spontaneous curvatures is straightforward. This produces a functional closely related to the Willmore energy \cite{Wil65}. 
Since we do not consider any topological changes of the membrane, the Gaussian curvature term,
which is associated to the Euler characteristic, is a constant, and can therefore be neglected.
That is we assume that the undeformed membrane is described by a surface that minimizes the energy functional 
\begin{equation}
\mathcal{W}(\Gamma) := \int_{\Gamma} \frac{1}{2}\kappa H^2 + \sigma \;do
\label{W_energy}
\end{equation}
subject to some constraints introduced below.  

\subsubsection*{Related work}
The Helfrich energy functional is often studied in the Monge gauge, when $\Gamma$ may be para\-metrised as a graph
\begin{equation*} \label{mongeGauge}
\Gamma :=\left\{ (x_1,x_2,u(x_1,x_2)) \; \big| \; (x_1,x_2) \in \Omega \right\}.
\end{equation*}         
Here, $\Omega \subset \mathbb{R}^2$ is some flat reference domain and $u$ is the displacement function.
Within this approach $u$ is usually supposed to be suitably small. 
The Monge gauge was recently studied in\cite{EllGraHobKorWol15}. The authors introduced a general mathematical framework for deformations of flat surfaces. 
In particular, the behaviour of point forces and point constraints was completely characterised. 
Moreover, the authors proved well-posedness for models related to embedded inclusions and examined their interactions. 
These effects were illustrated in numerical results obtained from finite element methods introduced to solve the related PDEs.      

In this paper, we will study small displacements from closed surfaces. The main application will be spheres and tori, although the results are more general. 
We will derive an energy functional which is a second order approximation of the Helfrich energy \eqref{H_energy}. 
In the case of a sphere such an approximation has been computed directly in \cite{Hel86,Saf83} in order to study shape fluctuations of approximately spherical vesicles 
and microemulsion droplets, respectively. Certainly, the advantage of our derivation is that it can also be applied to study deformations from other surfaces, 
for example, from tori. Once we have derived an appropriate energy functional, we may use the same theory developed in \cite{EllGraHobKorWol15}, which was done in a general Hilbert space setting. 

The effects of point forces on the membrane have been studied in Biophysics literature \cite{GovGop06,VekGov07}. 
These papers study a dynamical problem related to the diffusion of activator proteins which induce actin polymerisation. 
This work is done in the Monge gauge and also accounts for spontaneous curvature of the embedded activator proteins. 
In \cite{GovGop06} it is observed that the membrane motion can be wave like or lead to filipodia formation, depending on the sign of the spontaneous curvature. 
In \cite{VekGov07} it is observed that clusters of curvature-inducing activator proteins undergo a phase separation (aggregation to one region). 
This causes the formation of filipodia (membrane protrusions) as the activator proteins induce actin polymerisation. 
Our work will differ from this in a number of key respects. The primary difference is we do not work in the Monge gauge but instead on closed surfaces. 
Secondly, we will consider individual point couplings with the membrane, rather than accounting for a density of couplings.
We do this as we aim to describe the membrane mediated interactions between the couplings rather than a more macro level study of the entire membrane. 
One of our primary objectives is to establish a mathematical framework in which we can apply abstract results for well-posedness 
and produce numerical computations based on surface finite elements.   

\subsubsection*{Outline of the Paper}   
We begin by introducing the appropriate notation for the formulation of surface PDEs in Section \ref{Section_notation}. 
We recapitulate surface calculus, the definition of tangential derivatives and of the Laplace-Beltrami operator. 
This allows us to define the Sobolev spaces on which we will pose our minimisation problems. 
In Section \ref{sec:zeroTension} we introduce an energy which
is based on the Helfrich functional and the work done by forces applied to the surface in the zero surface tension case. We also discuss certain deformations that we will not permit.
Apart from physical reasons this will ensure that the mathematical problems in Sections \ref{Section_point_forces} and \ref{Section_numerics} are well-posed. In Section \ref{sec:positiveTension} our approach is extended in order to incorporate positive surface tension on a spherical surface.
In Section \ref{Section_point_forces} we will formulate energy minimisation problems associated with the energy
in such a way that we may apply general Hilbert space theory for their well-posedness. 
In particular, we will focus on point forces and point displacements.
We conclude by discussing the surface finite element methods used to numerically solve the related surface PDEs in Section \ref{Section_numerics}. 
We present numerical simulations detailing the membrane mediated interactions between point forces and examining the effect of point constraints.

%%%%%%%%%%%%%%%%%% SECTION %%%%%%%%%%%%%

\section{Notation and preliminaries} %$$ SECTION %%%
\label{Section_notation}
In the following we consider an embedding $x: \mathcal{M} \rightarrow \mathbb{R}^{3}$
of a two-dimensional connected, closed (that is compact and without boundary), orientable manifold
(that is a topological space which is locally homeomorphic to open
subsets $\Omega_i$ of $\mathbb{R}^2$ via the so-called coordinate charts 
$\mathcal{C}_i : U_i \subset \mathcal{M} \rightarrow \Omega_i \subset \mathbb{R}^2$). 
In the following we assume that $\mathcal{M}$ and $x$ are as regular as needed, but at most of class $C^4$.
The image $\Gamma := x(\mathcal{M})$ of $\mathcal{M}$ is a two-dimensional connected, closed, 
orientable hypersurface embedded into $\mathbb{R}^{3}$.

\iffalse
The reason for introducing the reference manifold $\mathcal{M}$ is that $\Gamma$ or respectively, 
the embedding $x$, is not known a priori. Rather, the shape of the undeformed membrane, that is $\Gamma$,
is determined by minimising the energy functional (\ref{W_energy}) subject to some additional constraints,
which we will introduce below. 
This means that the reference manifold $\mathcal{M}$ is fixed, whereas the embedded hypersurface $\Gamma$
depends on the energy functional and the imposed constraints.
Nevertheless, it will turn out below that it is possible to give an exact description of $\Gamma$. 
For example, for a genus zero surface, $\Gamma$ will be a sphere of a problem dependent radius.
In contrast, the shape of the deformed membrane, which is given by the minimizer
of a generalized energy functional, cannot be stated in such terms. We are therefore dependent
on numerical simulations in order to study the deformed membranes. 
Since it is physically reasonable to assume that the membrane deformations are indeed small,
we can use the information about the undeformed membrane
in order to obtain a problem, which is both analytically and computationally more amenable. 
We will illuminate this approach in the next section.
\fi

Henceforward, the Euclidean scalar product is denoted by $v \cdot w := v_\alpha w_\alpha$,
where we have made use of the convention to sum over repeated indices. For matrices 
$A,B \in \mathbb{R}^{3 \times 3}$ we define the scalar product 
$A:B := A_{\alpha \beta} B_{\alpha\beta}$. 

According to the Jordan-Brouwer separation theorem, see \cite{Lima}, there exists a bounded domain $D$ which
has $\Gamma$ as its point set boundary. The unit normal $\nu$ to $\Gamma$  
that points away from this domain is called the outward unit normal.
We define $P:=  \unit - \nu \otimes \nu$ on $\Gamma$ 
to be, at each point of $\Gamma$, the projection onto the corresponding tangent space.
Here $\unit$ denotes the identity matrix in $\mathbb{R}^{3}$. 
For a differentiable function $f$ on $\Gamma$ we define the tangential gradient by
\begin{equation*}
	\nabla_{\Gamma} f := P \nabla \overline{f},
\end{equation*}
where $\overline{f}$ is a differentiable extension of $f$ to an open neighbourhood of 
$\Gamma \subset \mathbb{R}^{3}$. Here, $\nabla$ denotes the usual gradient in $\mathbb{R}^{3}$.
The above definition only depends on the values of $f$ on $\Gamma$.
In particular, it does not dependent on the extension $\overline{f}$, 
see Lemma 2.4 in \cite{DziEll13} for more details. 
The components of the tangential gradient are denoted by $(\D 1 f, \D 2 f, \D {3} f)^T := \nabla_\Gamma f$.
For a twice differentiable function the Laplace-Beltrami operator is defined by
$$
	\Delta_\Gamma f := \nabla_\Gamma \cdot \nabla_\Gamma f.
$$

The extended Weingarten map $\mathcal{H} := \nabla_{\Gamma} \nu$ is symmetric and has zero eigenvalue
in the normal direction. The eigenvalues $\kappa_i$, $i=1, 2$, belonging to the tangential eigenvectors
are the principal curvatures of $\Gamma$. The mean curvature $H$ is the sum of the principal curvatures,
that is $H := \sum_{i=1}^2 \kappa_i = \mbox{trace}\;( \mathcal{H}) = \nabla_{\Gamma} \cdot \nu$. 
Note that our definition differs from the more common one by a factor of $2$.
We will denote the identity function on $\Gamma$ by $id_\Gamma$, that is $id_\Gamma(p) = p$ for all $p \in \Gamma$.
The mean curvature vector $H \nu$ satisfies $H \nu = - \Delta_\Gamma id_\Gamma$, see Section 2.3 in \cite{DecDziEll05}.
Tangential gradients satisfy the following commutator rule, see Lemma 2.6 in \cite{DziEll13},
\begin{equation}
\underline{D}_\alpha \underline{D}_\beta f - \underline{D}_\beta\underline{D}_\alpha f = \left(\mathcal{H} \nablaG f \right)_\beta\nu_\alpha - \left(\mathcal{H} \nablaG f \right)_\alpha\nu_\beta . 
\label{commutator_rule}
\end{equation}
The Sobolev spaces $H^1(\Gamma)$ and $H^2(\Gamma)$ on the hypersurface $\Gamma$ are defined by
\begin{align*}
	& H^1(\Gamma) := \left\{ f \in L^2(\Gamma) \;|\; 
	\textnormal{$f$ has weak derivatives $\D \alpha f \in L^2(\Gamma)$, $\alpha = 1, 2, 3$}\right\}\\
	& H^2(\Gamma) := \left\{ f \in H^1(\Gamma) \;|\;
	\textnormal{all weak derivatives $\D \beta \D \alpha f \in L^2(\Gamma)$,
	$\alpha, \beta = 1, 2, 3$ exist}\right\},
\end{align*}
where $\D \alpha f := v_\alpha \in L^2(\Gamma)$ is said to be the weak derivative of $f$ if
$$
	\int_{\Gamma} f \D \alpha  \phi \;do = - \int_\Gamma v_\alpha \phi \;do 
	+ \int_{\Gamma} f \phi H \nu_\alpha \;do 
$$
for all smooth test functions $\phi$ on $\Gamma$, see also Definition 2.11 in \cite{DziEll13}.
Here, the integrals are taken with respect to the two-dimensional Hausdorff measure on $\Gamma$. The Sobolev space $H^1(\Gamma)$ is a Hilbert space when endowed with the standard $H^1$ inner product and induced norm,
\[
( u,v )_{H^1(\Gamma)} := \int_\Gamma \nablaG u \cdot \nablaG v + uv \;do \quad \text{and} \quad \|u\|_{H^1(\Gamma)}:= \sqrt{( u,u )_{H^1(\Gamma)}}. 
\]
Similarly, $H^2(\Gamma)$ is a Hilbert space when endowed with the following inner product and induced norm
\[
( u,v )_{H^2(\Gamma)} := \int_\Gamma \DeltaG u \DeltaG v  + \nablaG u \cdot \nablaG v + uv \;do \quad \text{and} \quad \|u\|_{H^2(\Gamma)}:= \sqrt{( u,u )_{H^2(\Gamma)}}. 
\] 
Observe we are not using the standard inner product and induced norm on $H^2(\Gamma)$ which contains mixed second order derivatives. On a closed surface however the norm we defined above is equivalent to the standard $H^2(\Gamma)$ norm, see \cite{DziEll13} for details.
 
We next assume that $\Gamma_s := x_s(\mathcal{M})$ depends on a parameter $s \in (-\delta, \delta)$,
$\delta > 0$. The material derivative $\dot{f}$ of a function 
$f : \bigcup_{s \in (-\delta, \delta)} \Gamma_s \times \{ s \} \rightarrow \mathbb{R}$ 
is then defined by
\begin{equation}
	\dot{f} := \frac{\partial(f \circ x_s)}{\partial s} \circ x^{-1}_s.
	\label{material_derivative}
\end{equation}
We will also use the notation $\dot{\partial} f$ to denote the material derivative.
The transport formula, see Theorem 5.1 in \cite{DziEll13}, states that
\begin{equation}
	\frac{d}{ds} \int_{\Gamma_s} f \;do_s = \int_{\Gamma_s} \dot{f} + f \; \nabla_{\Gamma_s} \cdot V \;do_s,
	\label{transport_formula}
\end{equation}
where the vector field $V$ on $\Gamma_s$ is given by $V \circ x_s := \frac{\partial}{\partial s} x_s$.
If $X(\theta)$ is a local parametrization of $\Gamma$, see \cite{DziEll13} section 2, the first fundamental form $G(\theta):=(g_{ij}(\theta))_{i,j=1,2}$ has entries
\[
g_{ij}(\theta) := \frac{\partial X}{\partial \theta_i}\cdot \frac{\partial X}{\partial \theta_j}.
\]
The matrix $G$ is invertible and we denote the entries of $G^{-1}$ by $g^{ij}$.

%%%%%%%%%%%%%%%%%% SECTION %%%%%%%%%%%%%

\section{Modelling of small surface deformations without surface tension} %% SECTION %%%
\label{sec:zeroTension}
\subsection{Deformations due to small external forces}
\label{Section_deformations_due_to_forces}

We begin by considering surfaces without tension, that is we model their energy by the functional $\mathcal{W}$ in \eqref{W_energy} with $\sigma = 0$. We will fix the constant $\kappa=1$ since its only effect is to rescale the energy in this case. The resulting energy functional is the Willmore functional and here we denote it by $W$, that is 
\begin{equation}
W(\Gamma) := \int_\Gamma \frac{1}{2}H^2 \;do.
\end{equation}

In the following we will consider surfaces which are critical points for the energy $W$, we will denote these by $\Gamma_0$. This means that $\Gamma_0$ satisfies 
\begin{equation} \label{eq:WillmoreSurface}
\frac{d}{ds}W(\Gamma_s) \bigg|_{s=0} = 0, \qquad \forall u \in C^2(\Gamma_0).
\end{equation} 
where $\Gamma_s:=\left\{ p + s u(p)\nu_0(p) \;|\; p \in \Gamma_0 \right\}$. We will refer to $\Gamma_0$ as an undeformed surface. 
We now discuss how small deformations that are due to some small external forces can be incorporated into this model by perturbing the energy functional.

The undeformed membrane $\Gamma_0$ is now exposed to some external forces.
Since the exact form of the forces is
negligible in this section, we describe them by some arbitrary (non-linear) energy functional
$\tilde{\mathcal{F}}(\Gamma)$. 
In Section \ref{Section_point_forces} we will discuss point forces in detail. Such forces can indeed be regarded as a model for forces acting on biomembranes in living cells.
We say that a force is small if the associated energy functional is small compared to the change in the bending energy. 
In this case we rescale the functional $\tilde{\mathcal{F}}$ by a small parameter $\varepsilon > 0$,
that is $\mathcal{F} := \tilde{\mathcal{F}} / \varepsilon$, 
such that the rescaled energy $\mathcal{F}$ is of the same
order as the change in the bending energy. For such forces the total energy of the membrane is given by
\begin{equation}
	\Je (\Gamma):=W(\Gamma) - \varepsilon \mathcal{F}(\Gamma)
	\label{energy_with_force_term}
\end{equation}
We are motivated by attempting to minimise this energy. Since the energy associated to the external forces is of order $\varepsilon$ we regard this as a perturbation of the Willmore energy $W$. It is then reasonable to assume that the deformation is also of order $\varepsilon$ and that a deformed surface $\Gamma$ can be described as a graph over $\Gamma_0$, explicitly deformed surfaces are of the form
\begin{equation} \label{graphDeformedSurface}
	\Gamma_\varepsilon := \{ p + \varepsilon (u \nu_0)(p) ~|~ p \in \Gamma_0 \},
\end{equation}
where the height function $u \in C^2(\Gamma_0)$ is defined on the undeformed membrane 
$\Gamma_0$. It describes the deformations of $\Gamma$ in the normal direction 
that are induced by the external forces.
We wish to find the deformed surface $\Gamma_\varepsilon$ for which the energy \eqref{energy_with_force_term} is least.
In the following we aim to find a good approximation 
for the above energy which will simplify energy minimisation to a linear PDE. We first note that the energy $\Je$
can be interpreted as a functional for the height function $u$ which depends on a scale parameter $\varepsilon$. With a slight abuse of notation, we therefore write $\Je(u)$
instead of $\Je(\Gamma_\varepsilon)$ in the following. 
We now treat $\Je(u)$
as a function of a single variable $\varepsilon$ and produce the following second order Taylor expansion 
\begin{equation}\label{genTaylor}
\Je(u) = \mathcal{J}_0(u)  
+ \varepsilon \frac{\mathrm{d}\Je(u)}{\mathrm{d}\varepsilon}\bigg|_{\varepsilon=0}   
+ \frac{\varepsilon^2}{2}  \frac{\mathrm{d}^2\Je (u)}{\mathrm{d}\varepsilon^2}\bigg|_{\varepsilon=0}  
+ O(\varepsilon^3).
\end{equation}
We observe that the first term $\mathcal{J}_0(u)=W(\Gamma_0)$ does not depend on $u$.
Since $\Gamma_0$ is assumed to be a critical point of $W$, the second term reduces to $-\mathcal{F}(\Gamma_0)$ and thus does not depend on $u$. 
The second order term is therefore the lowest order term that depends on $u$.
The Taylor expansion hence can be written as
\[
\Je(u) = W(\Gamma_0) -\varepsilon\mathcal{F}(\Gamma_0)  + \varepsilon^2 J(u) + O(\varepsilon^3).
\] 
with
\begin{equation}
	J(u) := \frac{1}{2}\frac{\mathrm{d}^2\Je(u)}{\mathrm{d}\varepsilon^2}\bigg|_{\varepsilon=0}.
	\label{definition_of_energy_J}
\end{equation}
The approximate energy $J(u)$ is the sum of variations of the functionals $W$ and $\mathcal{F}$. 
To derive an explicit formula for $J(u)$ will be part of the next section. 
Instead of determining minimisers of the original energy in (\ref{energy_with_force_term}), 
we aim to approximate them by considering the novel energy $J$.

\subsection{Derivation of an energy functional for the height function}
\label{Section_derivation_of_approx_functional}
In this section we discuss the first and second variations of $W(\Gamma)$. We first consider variations $\Gamma_\varepsilon := \left\{ p + \varepsilon (u \nu)(p) ~|~ p \in \Gamma \right\}$
on arbitrary surfaces $\Gamma$ before we restrict the results to $\Gamma_0$.
We will then use these results in (\ref{definition_of_energy_J}) to obtain an explicit formula
for $J(u)$. In the next sections we will discuss the application of this result to the deformation of spheres and Clifford tori. 

\subsubsection*{First and second variations} 
In the following we will usually state the results of the first and second variations without using integration by parts 
and we will note it explicitly if a formula is based on integration by parts.  
Whilst this distinction is minor in the present paper, there are several reasons why it is quite useful to separate these results from each other:
\begin{enumerate}
\item[1.] Using integration by parts requires higher regularity of the surface $\Gamma$ or of the embedding $x$, respectively.
\item[2.] The  approach presented here could be extended to surfaces with boundary, in which case boundary conditions have to be taken into account.
To consider surfaces with boundary might indeed be interesting in order to model finite size inclusions in biomembranes.
\item[3.] On piecewise linear interpolations of the surface $\Gamma$, see Section \ref{Section_numerics}, 
integration by parts would lead to additional terms depending on the discontinuous co-normals of the mesh simplices.
This means that the discretization of formulas, which are equivalent in the smooth case, usually leads to different algorithms.
Therefore, one has to be very careful in Numerics, which formula one chooses for the discretization. 
\end{enumerate}
To compute the required derivatives of $W$ we use the following formulae relating them to variations of $W$.
\begin{align*}
& W'(\Gamma)[u\nu] := \frac{\mathrm{d}W(\Gamma_\varepsilon)}{\mathrm{d}\varepsilon}\bigg|_{\varepsilon =0}  \\
& W''(\Gamma)[u\nu,u\nu] := \frac{\mathrm{d^2}W(\Gamma_\varepsilon)}{\mathrm{d}\varepsilon^2}\bigg|_{\varepsilon =0}
\end{align*}
\begin{remark}
In the following we assume sufficient smoothness of $u$ and $\Gamma$, however we require at most $C^4$ regularity.
The functionals, which we will obtain below from the second variation, can then be extended to $u \in H^2(\Gamma)$ using density arguments. 
\end{remark}
\begin{remark} \label{remark:firstAndSecondVariaitions}
By definition the first and second variation of a functional $F$ at the point $p$ in the direction of $v$ is
\begin{align*}
    & F'(p)[v] := \frac{d \phi(\varepsilon)}{d\varepsilon}\bigg|_{\varepsilon=0},
    \quad \textnormal{and} \quad
	 F''(p)[v,v] := \frac{d^2 \phi(\varepsilon)}{d\varepsilon^2}\bigg|_{\varepsilon=0},
\end{align*}  
where $\phi(\varepsilon) := F(p + \varepsilon v)$, see, for example, page 688 ff. in \cite{Zeidler}.
However, note that the second variations presented below and in the Appendix are based on the variation of the first variation, that is on
$$
	\frac{d \varphi(\mu)}{d\mu}\bigg|_{\mu=0} \quad \textnormal{with}
	\quad \varphi(\mu) := (F'[v])(p + \mu w).
$$
So, the question is whether this gives the correct expression for the second variation of the considered functionals.
In general this is not quite clear since $\varphi(\mu) \neq \frac{d \phi}{d \varepsilon}\big|_{\varepsilon = \mu}$. This condition will hold for each of our applications however.
For example, for the functional $F=W$, we have
\begin{align*}
	\varphi(\mu) = W'(\Gamma_\mu)[u^\ell \nu^\mu],
	\quad \textnormal{whereas} \quad
	\frac{d \phi}{d \varepsilon}\bigg|_{\varepsilon = \mu} = W'(\Gamma_\mu)[u^\ell \nu^\ell].
\end{align*} 
Here, $\Gamma_\mu := \left\{ c_\mu(p) ~|~ p \in \Gamma \right\}$ is the deformed surface with outward unit normal $\nu^\mu$, 
where $c_\mu$ is defined by $c_\mu := id_{\Gamma} + \mu u \nu$, and 
$u^\ell := u \circ c_\mu^{-1}$ is the lift of $u$ onto $\Gamma_\mu$. Note that 
$\nu^\ell := \nu \circ c_\mu^{-1}$ is the lift of the outward unit normal $\nu$ to $\Gamma$ onto $\Gamma_\mu$.
Using the embedding $x: \mathcal{M} \rightarrow \Gamma$ for $\Gamma$, an embedding $x_\mu$ for $\Gamma_\mu$ is given by $x_\mu := c_\mu \circ x$.
The material derivative (\ref{material_derivative}) with respect to $x_\mu$ of $u^\ell$ and $\nu^\ell$ is, in fact, zero.
On the other hand, the material derivative of $\nu^\mu$ usually does not vanish. 
We therefore find that
$$
	\frac{d^2 \phi(\varepsilon)}{d\varepsilon^2} \bigg|_{\varepsilon =0} - \frac{d \varphi (\mu)}{d \mu} \bigg|_{\mu =0 } 
	= - W'(\Gamma)[u \dot{\partial} \nu].	
$$
Since the material derivative $\dot{\partial} \nu$ is a tangent vector field to $\Gamma$, it follows from the invariance of $W(\Gamma)$ under diffeomorphisms,
that the first variation of $W(\Gamma)$ in the direction of $\dot{\partial} \nu$ vanishes.
We hence obtain that
$$
	\frac{d^2 \phi(\varepsilon)}{d\varepsilon^2} \bigg|_{\varepsilon =0} = \frac{d \varphi (\mu)}{d \mu} \bigg|_{\mu =0 }.
$$ 
The same applies to the other functionals considered in the text.
\end{remark}
The following results hold on arbitrary (sufficiently smooth) surfaces $\Gamma$ (with or without boundary). 
\begin{align}
\label{willmore1var}
& W'(\Gamma)[u\nu]= \int_\Gamma - H \left( \DeltaG u +|\mathcal{H}|^2u -\frac{1}{2}H^2u \right) \; do, \\
\label{willmore2var}
\begin{split}
& W''(\Gamma)[u\nu, g\nu] = \int_\Gamma (\DeltaG g + |\mathcal{H}|^2g)(\DeltaG u + |\mathcal{H}|^2u) +2H \mathcal{H}:(g\nablaG \nablaG u + u\nablaG \nablaG g )  \\
& \quad +2H \mathcal{H}\nablaG u \cdot \nablaG g + Hg \nablaG u \cdot \nablaG H - H^2 \nablaG u \cdot \nablaG g -\frac{3}{2}H^2 u \DeltaG g - H^2 g \DeltaG u \\
& \quad + \left( 2HTr(\mathcal{H}^3) -\frac{5}{2}H^2|\mathcal{H}|^2 + \frac{1}{2}H^4 \right) gu \;do 
\end{split}
\end{align}
This follows from (\ref{Willmore_functional_1var}) and Theorem \ref{Theorem_2var_Willmore} in the appendix.
In order to derive the above formula we have not used integration by parts.
It might therefore also be applied to surfaces with boundary.
We here restrict to closed surfaces. Integration by parts then gives, see also Remark \ref{Remark_integration_by_parts_for_Willmore},
\begin{align}\label{willmore2varIBP}
\begin{split}
& W''(\Gamma)[u\nu, g\nu] = \int_\Gamma (\DeltaG g + |\mathcal{H}|^2g)(\DeltaG u + |\mathcal{H}|^2u) +2H \mathcal{H}:(g\nablaG \nablaG u + u\nablaG \nablaG g )  \\
& \quad +2H \mathcal{H}\nablaG u \cdot \nablaG g - \frac{3}{2} H^2 \nablaG u \cdot \nablaG g -\frac{3}{2}H^2 ( u \DeltaG g + g \DeltaG u )\\
& \quad + \left( 2HTr(\mathcal{H}^3) -\frac{5}{2}H^2|\mathcal{H}|^2 + \frac{1}{2}H^4 \right) gu \;do 
\end{split}
\end{align}
The first variation of the Willmore energy can be found in \cite{Wil93}. The formula for the second variation 
was obtained in \cite{Gla11}. For the sake of completeness, we also present its derivation 
in the Appendix, since it is indeed a crucial part of this paper. Recall that we assume $\Gamma_0$ is chosen so that the first variation term vanishes, see \eqref{eq:WillmoreSurface}. 

To complete the calculation of $J(u)$, defined in \eqref{definition_of_energy_J}, we require the second derivative of the force term.
\[
\frac{\mathrm{d}^2 (\varepsilon\mathcal{F}(\Gamma_\varepsilon))}{\mathrm{d}\varepsilon^2}\bigg|_{\varepsilon=0} = 2 \mathcal{F}'(\Gamma_0)[u\nu]
\]
Here we have applied the definition of the first variation given in Remark \ref{remark:firstAndSecondVariaitions}.
The functional $J$ is a novel quadratic energy with which we will formulate the variational problems related to the surface displacement.

\begin{definition}
Given a surface $\Gamma_0 \subset \mathbb{R}^{3}$, we define the quadratic surface energy $J:H^2(\Gamma_0) \rightarrow \mathbb{R}$ by
\begin{align*}
J(u):&=\frac{1}{2} \frac{\mathrm{d}^2\Je(u)}{\mathrm{d}\varepsilon^2}\bigg|_{\varepsilon=0} = \frac{1}{2}W''(\Gamma_0)[u\nu_0,u\nu_0] - \mathcal{F}'(\Gamma_0)[u].
\end{align*}
\end{definition} 

Under the assumption $\Gamma_0$ is chosen such that the first variation $W'(\Gamma_0)$ vanishes, by the Taylor expansion \eqref{genTaylor}, 
$J(u)$ is an $O(\varepsilon^3)$ order approximation of $\Je(u)$, up to an additive constant.
\begin{lemma}
For an undeformed surface $\Gamma_0 \subset \mathbb{R}^{3}$ chosen such that $W'(\Gamma_0)$ vanishes we have
\[
\Je(u) = W(\Gamma_0) - \varepsilon\mathcal{F}(\Gamma_0) +\varepsilon^2 J(u) + O(\varepsilon^3).
\] 
\end{lemma}
Henceforward, we will neglect the constant and the $O(\varepsilon^3)$ terms. 
We interpret $J$ as a new energy. We aim to minimize this energy in the next sections.
This is of course only possible if the total energy is bounded from below. Since we want to determine minimizers by considering the associated variational problems,
and in particular, compute numerical approximations, we limit the space of admissible variations so that the bilinear form corresponding to the second variation term $W''(\Gamma_0)[\cdot \nu,\cdot \nu]$
is coercive in the $H^2(\Gamma_0)$-norm over this space. Note that in \cite{Sko15}, it was recently proved 
that Willmore immersions are local minimizers if the second variation of the Willmore functional 
is positive semi-definite with kernel equal to the sum of the space of infinitesimal M\"obius transformations and of the space of tangential variations. In our setting $\Gamma_0$ is a sphere or a Clifford torus and the latter condition is satisfied. Furthermore we only consider normal variations, thus there is a space of admissible variations with finite codimension over which we are able to formulate well-posed problems.   

\subsection{The kernel of $W''(\Gamma_0)$ in the cases of a sphere and a Clifford torus} 
We now examine the undeformed surfaces $\Gamma_0 = S(0,R)$, a sphere with radius $R$ centred at the origin and $\Gamma_0 = T(R,R\sqrt{2})$, a Clifford torus with tube radius $R$ centred at the origin. Both of these surfaces are Willmore surfaces, that is $W'(\Gamma_0)$ vanishes. The quadratic surface energy is given by
\[
J(u,\mu) = \frac{1}{2} a(u,u) - \mathcal{F}'(\Gamma_0)[u],
\]     
where we have introduced the bilinear form $a:H^2(\Gamma)\times H^2(\Gamma) \rightarrow \mathbb{R}$ defined by
\begin{equation} \label{eq:aFunctionalDefn}
a(u,v) := W''(\Gamma_0)[u\nu,v\nu].
\end{equation}
The bilinear form is bounded, symmetric and positive semi-definite, for the sphere this is immediate from Corollary \ref{Corollary_2var_Willmore_on_sphere} in the appendix, for the Clifford torus see \cite{MonNgu14,Wei78}. As remarked above, to formulate well-posed problems we will work in a subspace of $H^2(\Gamma)$ over which $a(\cdot,\cdot)$ is coercive. To find such a subspace one must first identify the kernel of $a$, that is the set
\[
Ker(a):= \left\{ v \in H^2(\Gamma_0) ~|~ a(v,w)=0 \;\forall w \in H^2(\Gamma_0) \right\}.
\]
For both the sphere and Clifford torus the kernel is finite dimensional, we will identify a basis in each case, that is we will write the kernel in the form
\[
Ker(a) = sp\left\{ f_1,...,f_M \right\},
\] 
for some $M:= dim(Ker(a)) \geq 1$.
This is done in the following lemma.

\begin{lemma} \label{lem:Kera} \ 
Let $\nu_1, \nu_2, \nu_3: \Gamma_0 \rightarrow \mathbb{R}$ denote the components of the outward normal vector field $\nu_{\Gamma_0}$.
\begin{itemize}
\item If $\Gamma_0$ is a sphere or a Clifford torus then
\[
Ker(a) = Moeb(\mathbb{R}^3)\cdot\nu_{\Gamma_0} := \left\{ u \in H^2(\Gamma_0) ~|~ u(x)=f(x)\cdot\nu_{\Gamma_0}(x) \text{ for some } f \in Moeb(\mathbb{R}^3)  \right\}.
\] 

\item If $\Gamma_0$ is a sphere then 
\[
Ker(a)=sp\left\{1,\nu_1,\nu_2,\nu_3 \right\}.
\] 
\item If $\Gamma_0$ is a Clifford torus then
\begin{align*}
	Ker(a)&=sp\left\{\nu_1,\nu_2,\nu_3, f_4(x) := x_3\nu_1-x_1\nu_3, f_5(x) := x_3\nu_2-x_2\nu_3, f_6(x):= x\cdot\nu,\right. \\
	&\qquad \quad \left. f_7(x):= 2x_1(x\cdot\nu)-|x|^2\nu_1, f_8(x):= 2x_2(x\cdot\nu)-|x|^2\nu_2 
	 \right\}.
\end{align*}
\item If $\Gamma_0$ is a sphere or Clifford torus then there exists $C(\Gamma_0)>0$ such that
\[
a(v,v) \geq C(\Gamma_0) \|v\|_{H^2(\Gamma_0)} \quad \forall v \in Ker(a)^\perp,
\]
where $\perp$ denotes orthogonality with respect to the $H^2(\Gamma_0)$ inner product.
\end{itemize}
\end{lemma}
 
\noindent Here $Moeb(\mathbb{R}^3)$ denotes the set of infinitesimal M\"{o}bius transformations on $\mathbb{R}^3$. For an abstract definition of this set see \cite{Res89}, here we will use an equivalent characterisation, also presented in \cite{Res89}.

\begin{proof}
We begin with $\Gamma_0=S(0,R)$, a sphere. The bilinear form $a$ is explicitly calculated in the appendix as Corollary \ref{Corollary_2var_Willmore_on_sphere}. Here we consider $n=2$, hence the bilinear form is given by
\[
a(u,v) = \int_{\Gamma_0} \Delta_{\Gamma_0} u \Delta_{\Gamma_0} v - \frac{2}{R^2} \nabla_{\Gamma_0} u \cdot \nabla_{\Gamma_0} v \;do
\] 
Note that $1 \in Ker(a)$ and that, on a sphere, each component of the normal $\nu_i$ is an eigenfunction of $-\Delta_{\Gamma_0}$ with eigenvalue $2/R^2$, hence
\begin{equation} \label{sphereKernelInclusion}
Sp\left\{ 1,\nu_1,\nu_2,\nu_3 \right\} \subset Ker(a).
\end{equation}
To obtain equality in this inclusion and prove the coercivity statement for a sphere we will use the following Poincar\'{e} inequality.
\begin{equation} \label{PoincareIneq}
\int_{\Gamma_0} u^2 \;do \leq \frac{R^2}{6} \int_{\Gamma_0} |\nabla_{\Gamma_0} u|^2  \;do \leq \frac{R^4}{36} \int_{\Gamma_0} (\Delta_{\Gamma_0} u)^2  \;do \quad \forall u \in Sp\left\{ 1,\nu_1,\nu_2,\nu_3 \right\}^\perp,
\end{equation} 
where again $\perp$ denotes orthogonality with respect to the $H^2(\Gamma_0)$ inner product.
To prove this, note for a sphere of radius $R$ the negative Laplace-Beltrami operator, $-\DeltaG$, has eigenvalues
\[
\lambda_k=\frac{k(k+1)}{R^2} \quad \text{with multiplicities} \quad N_k=\binom{k+2}{2}, \quad k \in \mathbb{N}\cup \left\{ 0 \right\}.
\] 
See \cite{Shu87} for a proof on the unit sphere, from which we deduce the above. It follows that $\lambda_0=0$ and $N_0=1$, thus the zero eigenfunctions are simply constant functions. The next eigenvalue $\lambda_1=2/R^2$ has multiplicity $N_1 = 3$. We then see the $\lambda_1$-eigenfunctions are spanned by $\left\{ \nu_1,\nu_2,\nu_3 \right\}$. Thus the optimal Poincar\'{e} constant over $Sp\left\{ 1,\nu_1,\nu_2,\nu_3 \right\}^\perp$, $C_P$, satisfies
\[
C_P^{-2} = \inf_{v \in S} \frac{\int_\Gamma |\nablaG v|^2 \;do}{\int_\Gamma v^2 \;do} = \lambda_2 
\]
where $S:=\left\{ u \in H^1(\Gamma_0) \;|\; 0=\int_\Gamma u \;do = \int_\Gamma u \nu_i \;do, \: i=1,2,3 \right\}$ and $\lambda_2$ is the second non-zero eigenvalue for the negative Laplace-Beltrami operator. The validity of the inequality used above relies upon the fact that $H^2(\Gamma_0) \cap S = Sp\left\{ 1,\nu_1,\nu_2,\nu_3 \right\}^\perp$, that is for $H^2(\Gamma_0)$ functions, membership of $S$ encodes both $L^2$ and $H^2$ orthogonality to $Ker(a)$. To see this, a simple calculation shows
\begin{equation} \label{eq:L2H2OrthogSphere}
( u, 1 )_{H^2(\Gamma_0)} = (u,1)_{L^2(\Gamma_0)} \quad \text{and} \quad ( u, \nu_i )_{H^2(\Gamma_0)} = \left(\tfrac{4}{R^4} + \tfrac{2}{R^2} + 1 \right)
(u,\nu_i)_{L^2(\Gamma_0)} \quad \forall u \in H^2(\Gamma_0).
\end{equation}
For a sphere $\lambda_2=6/R^2$, proving the first inequality in \eqref{PoincareIneq}. The second inequality then follows by integration by parts and the H\"{o}lder inequality. 
%By a result of \cite{DziEll13} we have, for any $u \in H^2(\Gamma_0)$,
%\begin{equation} \label{H^2_seminorm_sphere}
%|u|^2_{H^2(\Gamma_0)} = \int_{\Gamma_0} (\Delta_{\Gamma_0} u)^2 -(H\mathcal{H}-2\mathcal{H}^2)\nabla_{\Gamma_0} u \cdot \nabla_{\Gamma_0} u \;do = \int_{\Gamma_0} (\Delta_{\Gamma_0} u)^2 \;do.  
%\end{equation} 
Coercivity of $a(\cdot,\cdot)$ over $Sp\left\{1,\nu_1,\nu_2,\nu_3 \right\}^\perp$ follows from \eqref{PoincareIneq}. 
%and \eqref{H^2_seminorm_sphere}. 
From this we deduce equality in \eqref{sphereKernelInclusion} and coercivity of $a(\cdot,\cdot)$ over $Ker(a)^\perp$ for a sphere.    

For $\Gamma_0$ a Clifford torus the first and final statements in this lemma are proven in \cite{MonNgu14,Sko15}. For the remaining two statements we use the characterisation of $Moeb(\mathbb{R}^3)$ given in \cite{Res89}.   
\[
f \in Moeb(\mathbb{R}^3) \text{ if and only if } f(x)=a+(K+\alpha \unit)x +2(b\cdot x)x -|x|^2b
\]  
where $\alpha \in \mathbb{R}$, $a,b \in \mathbb{R}^3$, $\unit \in \mathbb{R}^{3 \times 3}$ is the identity matrix and $K \in \mathbb{R}^{3 \times 3}$ is a skew-symmetric matrix. We may regard $Moeb(\mathbb{R}^3)$ as a 10 dimensional subspace of $C^\infty(\mathbb{R}^3)$ and can thus determine $Moeb(\mathbb{R}^3)\cdot \nu_{\Gamma_0}$ for each of our choices of $\Gamma_0$, using a suitable parametrisation of each surface.
\end{proof}
Note that the kernel is not 10 dimensional in either case. This is due to the fact that for some infinitesimal M\"{o}bius transformations $f$, $f(x)$ lies in the tangent plane to $\Gamma_0$ for each $x \in \Gamma_0$.
Working over $Ker(a)^\perp$ we are able to formulate well posed mathematical problems. One can justify this step physically as non admissible variations are those which alter the surface but do not change the Willmore energy $W$ up to second order.  

As we will work with $Ker(a)^\perp$ frequently it is useful to detail three methods by which this set can be characterised. Firstly, the standard definition gives
\[
Ker(a)^\perp = \left\{ v \in H^2(\Gamma_0) ~|~ a( v,w ) = 0 \; \forall w \in Ker(a) \right\}.
\]
Secondly, by writing $Ker(a)= sp \left\{ f_1,...,f_M \right\}$, it follows
\[
Ker(a)^\perp = \left\{ v \in H^2(\Gamma_0) ~|~ (v,f_i)_{H^2(\Gamma_0)}=0 \; \forall 1 \leq i \leq M \right\}.
\]
Finally, it follows from our choice of $H^2(\Gamma_0)$ inner product that
\[
Ker(a)^\perp = \left\{ v \in H^2(\Gamma_0) ~|~ (v,g_i)_{L^2(\Gamma_0)}=0 \; \forall 1 \leq i \leq M \right\},
\]
where $g_i:= (\Delta_{\Gamma_0}^2 - \Delta_{\Gamma_0} + 1)f_i$. This characterisation follows by integrating the $H^2(\Gamma_0)$ inner product by parts, notice each $f_i$ is sufficiently regular to permit this.

%%%%%%%%%%%%%%%%%% SECTION %%%%%%%%%%%%%

\section{A spherical membrane under tension} \label{sec:positiveTension} %%$ SECTION %%
We now consider a spherical membrane, $\Gamma_0 = S(0,R)$,  with positive surface tension $\sigma > 0$, with the surface energy functional $\mathcal{W}$ given in \eqref{W_energy}. Note that we will only consider spherical membranes when considering a positve surface tension. To formulate relevant minimisation problems we also introduce a fixed volume constraint. This is required as $W$, the Willmore energy, is scale invariant but the surface tension term proportional to $\sigma$ is not. 

The fixed volume constraint is  physically reasonable.
Biological membranes are usually semipermeable, which means that certain molecules or ions
cannot diffuse through the membrane, whereas this is possible for other molecules like for water.
If such a membrane is contained in an isotonic environment,
that is a solvent which has the same effective solute concentration as the solution enclosed by the
membrane, the volume enclosed by this membrane does not change. 

We therefore here assume that
the volume enclosed by our deformed hypersurface $\Gamma$ is a constant given by $V_0 > 0$.
Let $D \subset \mathbb{R}^3$ denote the bounded domain which has $\Gamma$ as its
point boundary set. Then the volume of $D$ is given
by 
$$
	| D | := \int_D 1 \; dx = \frac{1}{3}\int_{D} \nabla \cdot x \; dx
	= \frac{1}{3} \int_{\partial D} x \cdot \nu \; do = \frac{1}{3} \int_{\Gamma} id_{\Gamma} \cdot \nu \; do.
$$
The assumption of a fixed enclosed volume can hence be written as $V(\Gamma) = V_0$, where
$$
	 V(\Gamma) := \frac{1}{3} \int_{\Gamma} id_{\Gamma} \cdot \nu \; do. 
$$
We introduce this constraint into the energy functional via a Lagrange multiplier.
This yields to the following Lagrangian functional which will be the principal object of our study in this section  
\begin{equation} 
\label{exact_energy}
\mathcal{L}(\Gamma,\lambda) := \int_\Gamma \frac{1}{2} \kappa H^2 + \sigma \; do 
+ \lambda \left( V(\Gamma) - V_0 \right).  
\end{equation}
\begin{remark}
In the variational formulation the term associated with the Lagrange multiplier $\lambda$ corresponds to a constraining force, which in the above case can be interpreted as a hydrostatic pressure maintaining the volume constant.
\end{remark}
In this section we will consider a critical point for the Lagrangian $\mathcal{L}$, which we will denote by $(\Gamma_0,\lambda_0)$. This means that $(\Gamma_0,\lambda_0)$ satisfies 
\[
(A1)
\left\{
\begin{aligned}
&\frac{d}{ds}\mathcal{L}(\Gamma_0,\lambda_0 + s\mu) \bigg|_{s=0} = 0, \qquad \forall \mu \in \mathbb{R},  \\
&\frac{d}{ds}\mathcal{L}(\Gamma_s,\lambda_0) \bigg|_{s=0} = 0, \qquad \forall u \in C^2(\Gamma_0),
\end{aligned}
\right.
\]
where $\Gamma_s:=\left\{ p + s (u\nu_0)(p) \;|\; p \in \Gamma_0 \right\}$. We will refer to $\Gamma_0$ as an undeformed surface. 
As for the tensionless membrane we will consider small deformations that are due to some small external forces. They will be incorporated into this model by perturbing 
the Lagrangian.

\subsection{Deformations due to small external forces}
As for the tensionless membrane we consider arbitrary small forces $\varepsilon\mathcal{F}$ which give rise to deformed surfaces of the form $\Gamma_\varepsilon$ as given in \eqref{graphDeformedSurface}. We also assume it is possible to write the Lagrange multiplier associated with the deformed membrane as $\lambda_\varepsilon=\lambda_0 + \varepsilon \mu$ for some $\mu \in \mathbb{R}$. The perturbed Lagrangian is hence given by
\begin{equation} \label{lagrangian_with_force}
\mathcal{L}_\varepsilon(\Gamma_\varepsilon,\lambda_\varepsilon):= \mathcal{W}(\Gamma_\varepsilon) + \lambda_\varepsilon (V(\Gamma_\varepsilon) - V(\Gamma_0) ) - \varepsilon \mathcal{F}(\Gamma_\varepsilon). 
\end{equation} 

We are motivated by attempting to find critical points of this energy. The rationale for this is that if $\Gamma_\varepsilon$ minimises the perturbed Helfrich energy
\[
\mathcal{W}(\Gamma_\varepsilon) - \varepsilon \mathcal{F}(\Gamma_\varepsilon),
\]    
subject to the volume constraint $V(\Gamma_\varepsilon)=V(\Gamma_0)$, then there exists a $\lambda_\varepsilon \in \mathbb{R}$ such that $(\Gamma_\varepsilon,\lambda_\varepsilon)$ is a critical point for $\mathcal{L}_\varepsilon$. Similarly to the construction of $J$ in the tensionless case, we aim to find a good approximation for this Lagrangian for which the determination of critical points reduces to a linear PDE. To do so we perform a second order Taylor expansion in $\varepsilon$, using a slight abuse of notation $\mathcal{L}_\varepsilon(u,\mu)=\mathcal{L}_\varepsilon (\Gamma_\varepsilon,\lambda_\varepsilon)$.  
\begin{equation} \label{L_taylor_series}
\mathcal{L}_\varepsilon(u,\mu)= \mathcal{L}_0(u,\mu) + \varepsilon \frac{\mathrm{d}\mathcal{L}_\varepsilon(u,\mu)}{\mathrm{d}\varepsilon}\bigg|_{\varepsilon=0}   
+ \frac{\varepsilon^2}{2}  \frac{\mathrm{d}^2\mathcal{L}_\varepsilon (u,\mu)}{\mathrm{d}\varepsilon^2}\bigg|_{\varepsilon=0}  
+ O(\varepsilon^3).
\end{equation}
We observe that the first term $\mathcal{L}_0(u,\mu)=\mathcal{W}(\Gamma_0)$ does not depend on $u$ or $\mu$.
Since $(\Gamma_0,\lambda_0)$ is assumed to be a critical point of $\mathcal{L}$, the second term reduces to $-\mathcal{F}(\Gamma_0)$ and thus does not depend on $u$ or $\mu$. 
We therefore see that the lowest order term that depends on $u$ or $\mu$ is the second order term.
The Taylor expansion hence can be written as
\[
\mathcal{L}(u,\mu) = \mathcal{W}(\Gamma_0) -\varepsilon\mathcal{F}(\Gamma_0)  + \varepsilon^2 L(u,\mu) + O(\varepsilon^3),
\] 
with
\begin{equation}
	L(u,\mu) := \frac{1}{2}\frac{\mathrm{d}^2\mathcal{L}_\varepsilon(u,\mu)}{\mathrm{d}\varepsilon^2}\bigg|_{\varepsilon=0}.
	\label{definition_of_lagrangian_L}
\end{equation}
The approximate Lagrangian $L(u,\mu)$ is the sum of variations of the functionals $\mathcal{W}$, $V$ and $\mathcal{F}$. 
To derive an explicit formula for $L(u,\mu)$ will be part of the next section. 
Instead of determining critical points of the original Lagrangian in \eqref{lagrangian_with_force}, 
we aim to approximate them by considering the novel Lagrangian $L$.

\subsection{Derivation of a Lagrangian for the height function}
Similarly to the treatment of $J$ in the tensionless case we calculate $L$ in terms of variations of its constituent functionals. The second variation of $\mathcal{W}$ is calculated by combining the second variations of the Willmore functional $W$ and the area functional $A$, both given in the appendix. We take derivatives of the force term involving $\mathcal{F}$ as previously and also require variations of the volume functional $V$ which are also computed in the appendix. The functional $L$ is a novel quadratic Lagrangian with which we will formulate the variational problems related to surface displacement.

\begin{definition}
For the surface $\Gamma_0 \subset \mathbb{R}^{3}$ with associated Lagrange multiplier $\lambda_0 \in \mathbb{R}$, the quadratic surface Lagrangian $L:H^2(\Gamma_0) \times \mathbb{R} \rightarrow \mathbb{R}$ is given by
\begin{align*}
L(u,\mu):&=\frac{1}{2} \frac{\mathrm{d}^2\mathcal{L}_\varepsilon(u)}{\mathrm{d}\varepsilon^2}\bigg|_{\varepsilon=0} \\
& = \frac{1}{2}\mathcal{W}''(\Gamma_0)[u\nu_0,u\nu_0] + \frac{1}{2}\lambda_0 V''(\Gamma_0)[u\nu_0,u\nu_0] + \mu V'(\Gamma_0)[u\nu_0] - \mathcal{F}'(\Gamma_0)[u]. 
\end{align*}
\end{definition} 
 
Under the assumption $(\Gamma_0,\lambda_0)$ is chosen such that the first variation $\mathcal{L}'(\Gamma_0,\lambda_0)$ vanishes, by the Taylor expansion \eqref{L_taylor_series}, 
$L(u,\mu)$ is an $O(\varepsilon^3)$ order approximation of $\mathcal{L}_\varepsilon(u,\mu)$, up to an additive constant.
Henceforward, we will neglect the terms which do not depend on $u$ or $\mu$ and the $O(\varepsilon^3)$ terms.
In this case the only undeformed surface we consider is a sphere, $\Gamma_0 =S(0,R)$. Fixing the Lagrange multiplier $\lambda_0 = -\sigma H = -2\sigma /R$ ensures that $(A1)$ is satisfied. The linearised Lagrangian is given explicitly by
\[
L(u,\mu)= \frac{1}{2}\int_{\Gamma_0} \kappa (\Delta_{\Gamma_0} u)^2 
+\left( \sigma -\frac{2\kappa}{R^2} \right)|\nabla_{\Gamma_0} u|^2 -\frac{2\sigma}{R^2} u^2 + 2\mu u \;do - \mathcal{F}'(\Gamma_0)[u].
\] 

Similarly to the tensionless case, it is important to identify a subspace of $H^2(\Gamma_0)$ over which the bilinear form corresponding to the quadratic part of the Lagrangian is coercive. In this case the appropriate bilinear form is given by
\begin{equation} \label{eq:a_sigmaDefn}
a_\sigma (u,v):= \int_{\Gamma_0} \kappa \Delta_{\Gamma_0} u\Delta_{\Gamma_0} v 
+\left( \sigma -\frac{2\kappa}{R^2} \right)\nabla_{\Gamma_0} u \cdot \nabla_{\Gamma_0} v -\frac{2\sigma}{R^2} uv \;do
\end{equation} 
Notice that $a_\sigma$ is coercive over $Ker(a)^\perp$, where $Ker(a)=sp\left\{1,\nu_1,\nu_2,\nu_3 \right\}$, again we refer to orthogonality with respect to the $H^2(\Gamma_0)$ inner product but recall this is equivalent to orthogonality with respect to the $L^2(\Gamma_0)$ inner product in this case (see \eqref{eq:L2H2OrthogSphere}). 
Furthermore, $Ker(a)^\perp$ is the largest subspace of $H^2(\Gamma_0)$ over which it is coercive. Notice also that on such a subspace the term associated with the linearised Lagrange multiplier vanishes
\[
\int_{\Gamma_0} 2\mu u \;do = 0 \quad \text{for all } (u,\mu) \in Ker(a)^\perp \times \mathbb{R}.
\]  
We can thus pose variational problems in precisely the same manner for the tensionless case and for membranes under tension, simply by using the bilinear form $a$ or $a_\sigma$ as appropriate and working over the space $Ker(a)^\perp$. Studying such variational problems will be the subject of the next section. 

%%%%%%%%%%%%%%%%%% SECTION %%%%%%%%%%%%%

\section{Minimising the linearised Willmore functional with point force loading and point displacement constraints} %% SECTION %%%%%
\label{Section_point_forces}
In this section we will take the undeformed surface to be a sphere, $\Gamma = S(0,R)$, or a Clifford torus $\Gamma = T(R,R\sqrt{2})$. However, we will only consider problems involving surface 
tension, that is $\sigma>0$, on a sphere $\Gamma=S(0,R)$. Note we drop the subscript for the undeformed surfaces, simply referring to them as $\Gamma$.
As done for the flat case, see \cite{EllGraHobKorWol15} Section 7, we may study the interactions of the membrane with thin filaments. These filaments are anchored to the cytoskeleton. Their effects are modelled by applying a point force or point constraint to the membrane. We begin with point forces. 

\subsection{Point forces} 
We begin by studying the effect of point forces applied in the normal direction to $\Gamma$. Considering $N$ point forces at locations $X_1,...,X_N \in \Gamma$, which we will henceforth denote by $X:=(X_1,...,X_N) \in \Gamma^N$. We consider a functional $\mathcal{F}_X$ in \eqref{energy_with_force_term} giving rise to  $\mathcal{F}_X'(\Gamma)$, such that  
\begin{equation}
\mathcal{F}_X'(\Gamma)[u]:= \sum_{i=1}^{N}  \beta_i u(X_i).
\end{equation}
Here $\beta_i \in \mathbb{R}\setminus \left\{ 0 \right\}$ are constants related to the magnitudes of the forces, hence the force term measures the work done by the point forces. To emphasise the dependence of the resulting quadratic energy functional $J$ upon the locations of the point forces $X$ we will use the notation $\mathcal{E}:H^2(\Gamma)\times \Gamma^N \rightarrow \mathbb{R}$ defined by
\[
\mathcal{E}(u,X):=\frac{1}{2}a_\sigma(u,u) - \mathcal{F}'_X(\Gamma)[u].
\] 
Here $a_\sigma$ is the bilinear functional defined in \eqref{eq:a_sigmaDefn} for $\sigma > 0$ and the second variation of the Willmore functional, defined in \eqref{eq:aFunctionalDefn}, for $\sigma = 0$.
In light of the previous two sections, to formulate a well posed minimisation for $\mathcal{E}(\cdot,X)$ problem we work over the linear space
\[
V:=Ker(a)^\perp = \left\{ v \in H^2(\Gamma) ~|~ ( v,w )_{H^2(\Gamma)} = 0 \; \forall w \in Ker(a) \right\}.
\]
Here $Ker(a)$, which depends upon the choice of undeformed surface $\Gamma$, is constructed as in Lemma \ref{lem:Kera}. Note that we could attempt to pose a minimisation problem over the whole space $H^2(\Gamma)$ but this would need a compatibility condition on the linear functional $\mathcal{F}'_X$ for existence and constraints on $u$ for uniqueness. We now state the energy minimisation problem and give an equivalent variational form. 

\begin{problem}[Point forces at fixed locations]\ \\ \label{prob:fixed_forces}
For given $X\in \Gamma^N$ find $u_X \in V$ satisfying the two equivalent properties
\begin{enumerate}[label=(\alph*)]
\item $u_X$ minimises $\mathcal{E}(\cdot,X)$ on $V$,
\item $a_\sigma(u,v) = \sum_{k=1}^N \beta_k v(X_k)$ for all $v \in V$.
\end{enumerate}
\end{problem}  
\noindent Existence and uniqueness of a solution to this problem follows from the Lax-Milgram theorem. 
Note, for a sphere the projection out of $Ker(a)$ can be justified physically. The minimisation problem above has a corresponding PDE representation involving four Lagrange multipliers, as $Ker(a)=Sp\left\{1,\nu_1,\nu_2,\nu_3 \right\}$ is four dimensional for a sphere. The Lagrange multiplier associated with the constant function $1$ corresponds to a linearised hydrostatic pressure which enforces the volume constraint. Each of the three remaining Lagrange multipliers is associated with one of the components of the normal $\nu_i$. It can be shown that these multipliers correspond to linearised reaction forces which prevent $O(\varepsilon)$ translations of the centre of mass of the domain $D$ enclosed by $\Gamma$.

We also consider varying the locations $X \in \Gamma^N$, this problem can be stated in two equivalent forms.

\begin{problem}[Point forces with varying locations]\  \label{prob:varying_forces}
\begin{enumerate}[label=(\alph*)]
\item Find $(u,X) \in V \times \Gamma^N$ minimising the energy $\mathcal{E}(\cdot,\cdot) $ on $V \times \Gamma^N$. \label{varying_forces_handX}
\item Find $X \in \Gamma^N$ minimising the energy $X \mapsto \mathcal{E}(u_X,X) $ on $\Gamma^N$. \label{varying_forces_justX}
\end{enumerate}
\end{problem}
Existence of solutions for \ref{varying_forces_handX} follows by the same general theory that was applied in the flat case, see \cite{EllGraHobKorWol15} Proposition 9.6. Note \ref{varying_forces_handX} and \ref{varying_forces_justX} are equivalent in the sense that $(u,X)$ solves \ref{varying_forces_handX} if and only if $X$ solves \ref{varying_forces_justX} and $u=u_X$. We will thus refer to these equivalent minimisation problems simply as Problem \ref{prob:varying_forces}. 
\begin{proposition}
Without loss of generality assume $\beta_1 \leq \beta_2 \leq ... \leq \beta_N$.
\begin{description}
\item[Suppose $\mathbf{\beta_1 > 0}$ or $\mathbf{\beta_N < 0}$.] \hfill \\
$X \in \Gamma^N$ solves Problem \ref{prob:varying_forces} if and only if $X_i=X_0$ for all $i=1,...,N$ where $X_0 \in \Gamma$ is a solution to Problem \ref{prob:varying_forces} with parameters $\tilde{N}=1$ and $\tilde{\beta_1}=1$.
\item[Suppose $\mathbf{\beta_k <0}$ and $\mathbf{\beta_{k+1} >0}$ for some $\mathbf{1 \leq k \leq N-1}$.] \hfill \\
$X \in \Gamma^N$ solves Problem \ref{prob:varying_forces} if and only if both of the following hold.
\begin{enumerate}
\item $X_i=X^-$ for all $i=1,...,k$ and $X_i=X^+$ for all $i=k+1,...,N$.
\item $(X^+,X^-) \in \Gamma^2$ solves Problem \ref{prob:varying_forces} with parameters $\tilde{N}=2$ and \\
$\tilde{\beta} = \left(\sum_{i=1}^k \beta_i, \sum_{i=k+1}^N \beta_i \right)$.
\end{enumerate}
\end{description}
\end{proposition}
\begin{comment}
The proof follows by writing $u_X$ in the form
\[
u_X = \sum_{i=1}^N \beta_j \phi_{X_i}
\]
where, for $y \in \Gamma$, $\phi_y$ denotes the solution to Problem \ref{prob:varying_forces} with $N=1$, $X=(y)$ and $\beta_1=1$. We then apply the Cauchy Schwarz inequality, noting that $\phi_x,\phi_y$ are linearly dependent if and only if $x=y$. 
\end{comment}

\begin{proof}
For $y \in \Gamma$, let $\phi_y$ denote the solution to Problem \ref{prob:fixed_forces} with $N=1$, $X=(y)$ and $\beta_1=1$. By linearity it follows, for any $X \in \Gamma^N$,
\[
u_X = \sum_{i=1}^N \beta_j \phi_{X_i}.
\]
To prove the first statement, suppose $X_0 \in \Gamma$ is a solution to Problem \ref{prob:varying_forces} with parameters $\tilde{N}=1$ and $\tilde{\beta_1}=1$. Note that we have assumed $sign(\beta_1)=...=sign(\beta_N)$, hence for any $X \in \Gamma^N$,
\begin{align*}
\mathcal{E}(u_X,X) &= -\frac{1}{2}\sum_{i,j=1}^N\beta_i\beta_j a_\sigma(\phi_{X_i},\phi_{X_j}) \\
&\geq -\frac{1}{2}\sum_{i,j=1}^N\beta_i\beta_j a_\sigma(\phi_{X_i},\phi_{X_i})^{1/2}a_\sigma(\phi_{X_j},\phi_{X_j})^{1/2} \\
&\geq -\frac{1}{2}\sum_{i,j=1}^N\beta_i\beta_j a_\sigma(\phi_{X_0},\phi_{X_0}) = \mathcal{E}(u_{\tilde{X}},\tilde{X})
\end{align*}
where $\tilde{X}=(X_0,...,X_0)$. The first inequality used is the Cauchy Schwarz inequality and the second inequality follows from the definition of $X_0$. As these inequalities hold for any $X \in \Gamma^N$ we have proven the backwards implication.

Now suppose $X \in \Gamma^N$ solves Problem \ref{prob:varying_forces}, then in addition to the inequalities derived above it holds
\[
\mathcal{E}(u_{\tilde{X}},\tilde{X}) \geq \mathcal{E}(u_X,X),
\]  
hence equality holds at each step. Then, as we have equality in the Cauchy Schwarz inequalities used, $\phi_{X_i}$ and $\phi_{X_j}$ are linearly dependent for each $i,j$. It follows $X_1=X_2=...=X_N$ and for any $y \in \Gamma$, set $Y=(y,y,...,y) \in \Gamma^N$ then
\[
-\frac{1}{2} \sum_{i,j=1}^N \beta_i \beta_j a_\sigma(\phi_y,\phi_y)=\mathcal{E}(u_Y,Y) \geq \mathcal{E}(u_X,X)= -\frac{1}{2}\sum_{i,j=1}^N\beta_i\beta_j a_\sigma(\phi_{X_1},\phi_{X_1}). 
\]  
Hence $X_1 \in \Gamma$ is a solution to Problem \ref{prob:varying_forces} with parameters $\tilde{N}=1$ and $\tilde{\beta_1}=1$ and we have proven the forwards implication. 

For the second statement observe, for any $Y,Z \in \Gamma^N$,
\begin{align*}
\mathcal{E}(u_Y,Y) &= \mathcal{E}(u_Z,Z) - a_\sigma(u_Y-u_Z,u_Z) -\frac{1}{2}a_\sigma(u_Y-u_Z,u_Y-u_Z), \\ 
&= \mathcal{E}(u_Z,Z) - \sum_{i=1}^N \beta_i(u_Z(Y_i)-u_Z(Z_i)) -\frac{1}{2}a_\sigma(u_Y-u_Z,u_Y-u_Z). 
\end{align*}
Now suppose $X \in \Gamma^N$ solves Problem \ref{prob:varying_forces} and there exists $1\leq i <j \leq k$ such that $X_i \neq X_j$. Without loss of generality assume $u_X(X_j) \leq u_X(X_i)$. Let $X' \in \Gamma^N$ be given by $X'_l = X_l$ for $l \neq i$ and $X'_i=X_j$.  
Then using the above calculation we obtain 
\begin{equation} \label{eq:pointMovingDecreaseEnergy}
\mathcal{E}(u_{X'},X') = \mathcal{E}(u_X,X) -\beta_i (u_X(X_j)-u_X(X_i)) -\frac{1}{2}a_\sigma(u_{X'}-u_X,u_{X'}-u_X).
\end{equation}
As $X \neq X'$ it follows $u_{X} \neq u_{X'}$ and hence $a_\sigma(u_{X'}-u_X,u_{X'}-u_X) > 0$. Thus
\[
\mathcal{E}(u_{X'},X') < \mathcal{E}(u_X,X)
\]
which is a contradiction, hence $X_1=X_2=...=X_k=:X^-$. An identical argument shows $X_{k+1}=X_{k+2}=...=X_N=:X^+$. It follows
\[
u_X = \sum_{l=1}^k \beta_l \phi_{X^-} + \sum_{l=k+1}^N \beta_l \phi_{X^+}.
\] 
Now for any $(Y^+,Y^-) \in \Gamma^2$ let $Y \in \Gamma^N$ be such that $Y_l=Y^-$ for $1\leq l \leq k$ and $Y_l=Y^+$ for $k+1\leq l \leq N$. Then $\mathcal{E}(u_X,X) \leq \mathcal{E}(u_Y,Y)$ and using the above expression for $u_X$ we deduce that $(X^+,X^-)$ solves Problem \ref{prob:varying_forces} with parameters $\tilde{N}=2$ and $\tilde{\beta} = \left(\sum_{i=1}^k \beta_i, \sum_{i=k+1}^N \beta_i \right)$.

For the reverse implication in the second statement, suppose $Y \in \Gamma^N$. Using the same technique as in \eqref{eq:pointMovingDecreaseEnergy} we may form $Y' \in \Gamma^N$ such that $Y'_1=...=Y'_k$,  $Y'_{k+1}=...=Y'_N$ and $\mathcal{E}(u_{Y'},Y') \leq \mathcal{E}(u_Y,Y)$. Then, using the fact that $(X^+,X^-)$ solves the $N=2$ problem,
\[
\mathcal{E}(u_{X},X) \leq \mathcal{E}(u_{Y'},Y') \leq \mathcal{E}(u_Y,Y).
\]  
Hence $X$ solves Problem \ref{prob:varying_forces}.
\end{proof} 

Note there is no uniqueness for this problem in general. Indeed, for the problem on a sphere with $N=1$, every $X \in S(0,R)$ solves Problem \ref{prob:varying_forces} due to rotational symmetry. 

\subsection{Point value constraints}
We will now consider filaments which fix the location of the membrane at a point. To do so we consider the following perturbed energy functional, for a general surface $\tilde{\Gamma}$.
\[
\mathcal{W}_\delta (\tilde{\Gamma}):=\mathcal{W}(\tilde{\Gamma}) + \frac{1}{2\delta} \sum_{i=1}^N d(\tilde{\Gamma},y_i)^2 
\]
Here $\delta > 0$ is a small penalty parameter, $d$ denotes the signed distance to the surface $\tilde{\Gamma}$ and $y_1,...,y_N \in \mathbb{R}^3$ are the locations fixed by the filaments. Similarly to the point forces, we assume the additional term is a small perturbation to the Willmore functional so that the resulting deformed surface may be expressed in the form $\Gamma_\varepsilon$, a graph over the undeformed surface $\Gamma$, which is chosen to be a sphere or Clifford torus here. In this case the small perturbation assumption is justified when the locations $y_i$ can be expressed in the form
\begin{equation} \label{eq:smallPeturb_yi}
y_i = X_i + \varepsilon \alpha_i\nu(X_i),
\end{equation}  
for some $X \in \Gamma^N$ and $\alpha_i \in \mathbb{R}^N$. As $\nu$ is determined by the choice of undeformed surface $\Gamma$ and of unit length, it is equivalent to see this as a penalty method for $u$, the height function. Note that we could also formulate a similar problem without the penalty method by applying point constraints to the displacement $u$.  

The first problem we consider is for fixed locations. We begin by stating the general, non-linear problem which we aim to approximate.  

\begin{problem}[Point value constraints for $W$]\ \\ \label{prob:nonlinearCon}
Given $X \in \Gamma^N$ and $\alpha \in \mathbb{R}^N$ and $\delta>0$, find $u \in H^2(\Gamma)$ minimising $\mathcal{W}_\delta(\Gamma_\varepsilon(u))$.
\end{problem}

That is we wish to find the surface of the form $\Gamma_\varepsilon(u)= \left\{ p + \varepsilon (u\nu)(p) ~|~ p \in \Gamma \right\}$ which minimises the penalised Willmore energy $\mathcal{W}_\delta$. For a surface of the form $\Gamma_\varepsilon$, with the small perturbation assumption \eqref{eq:smallPeturb_yi}, this energy reads
\[
\mathcal{W}_\delta (\Gamma_\varepsilon) = \mathcal{W}(\Gamma_\varepsilon) + \frac{1}{2\delta} \sum_{i=1}^N d(\Gamma_\varepsilon, X_i + \varepsilon \alpha_i \nu(X_i) )^2.
\]   
We will use $(\varepsilon^2 /2)a_\sigma(\cdot,\cdot)$, the previously defined second order approximation to $\mathcal{W}(\Gamma_\varepsilon)$. We require a similar approximation of the distance function term. First notice
\[
d(\Gamma_\varepsilon,X_i + \varepsilon \alpha_i \nu(X_i)) \big|_{\varepsilon = 0} = d(\Gamma,X_i) = 0.
\]
Secondly, the first derivative with respect to $\varepsilon$ is given by
\begin{align*}
\frac{d}{d\varepsilon} d(\Gamma_\varepsilon, X_i + \varepsilon \alpha_i \nu(X_i) ) \bigg|_{\varepsilon=0} &= \nabla d(\Gamma,X_i) \cdot \frac{d}{d\varepsilon} (X_i+\varepsilon \alpha_i \nu (X_i)) \bigg|_{\varepsilon=0} + \dot{\partial}d(\Gamma_\varepsilon, X_i), \\
&= \alpha_i - u(X_i).
\end{align*}
The final line holds as $\nabla d(\Gamma,X_i) = \nu(X_i)$ and $\dot{\partial}d(\Gamma_\varepsilon, X_i) = -u(X_i)$, see \cite{HinRin04}. It follows, expanding in $\varepsilon$ as previously,
\[
d(\Gamma_\varepsilon, X_i + \varepsilon \alpha_i \nu(X_i) )^2 = \varepsilon^2 (u(X_i)-\alpha_i)^2  
+ O(\varepsilon^3).
\]
We thus minimise the penalised quadratic energy functional
\[
J_\delta (u):= \frac{1}{2}a_\sigma(u,u) + \frac{1}{2\delta} \sum_{i=1}^N (u(X_i)-\alpha_i)^2.
\]
When $\sigma > 0$ this functional is minimised subject to the constraint $\int_\Gamma u \;do= 0$, which is the linearised form of the fixed volume constraint.

Similarly to the point forces problem this minimisation is not well posed over $H^2(\Gamma)$ and we must identify an appropriate subspace over which the problem is well posed. In order to identify an appropriate subspace we first introduce the following notation for affine subspaces.
\begin{definition}
Suppose $Z \subset H^2(\Gamma)$ is a linear subspace, $X \in \Gamma^N$ and $\gamma \in \mathbb{R}^N$. We denote by $Z_{X,\gamma}$ the affine space given as follows.
\[
Z_{X,\gamma} := \left\{ z \in Z ~|~ z(X_i)=\gamma_i \;\forall i=1,...,N  \right\} 
\]
\end{definition}  
For the linearised problem we consider the space $U_\sigma \subset H^2(\Gamma)$ given by
\[
U_\sigma:= \begin{cases} \left( H^2_{X,0}(\Gamma) \cap Ker(a) \right)^\perp &\mbox{if } \sigma = 0, \\
\left\{ u \in \left( H^2_{X,0}(\Gamma) \cap Ker(a) \right)^\perp ~\big|~ \int_\Gamma u \;do= 0 \right\} & \mbox{if } \sigma > 0. \end{cases}
\]
Here $\perp$ again means orthogonality with respect to the $H^2$ inner product. In the $\sigma=0$ case, the space $U_\sigma$ is the largest subspace of $H^2(\Gamma)$ over which well-posedness is possible. If we used a larger subspace $Z \supset U_\sigma$ then there exist elements $0 \neq v_0 \in Z \cap \left( H^2_{X,0}(\Gamma) \cap Ker(a) \right) $. For such elements $J_\delta(u+v_0)=J_\delta(u)$ hence no uniqueness is possible. A similar argument shows $U_\sigma$ is the largest subspace of $\left\{ u \in H^2(\Gamma) ~|~ \int_\Gamma u \;do = 0 \right\}$ over which well-posedness is possible in the $\sigma > 0$ case. 
\begin{remark} \label{rem:constraintsGenericCase}
Notice that $Ker(a)$ is a finite dimensional space and $u \in H^2_{X,0}(\Gamma)$ satisfies $N$ conditions of the form $u(X_i)=0$. The generic case for $N > dim(Ker(a))$ is thus $H^2_{X,0}(\Gamma) \cap Ker(a) = \left\{ 0 \right\}$ and $U_\sigma=H^2(\Gamma)$ if $\sigma = 0$ and $U_\sigma=\left\{ u \in H^2(\Gamma) ~|~ \int_\Gamma u \;do = 0 \right\}$ if $\sigma > 0$.  
\end{remark}

We now state the quadratic minimisation problem that will be studied.

\begin{problem} \label{pointConPenMethod}
Given $X \in \Gamma^N$, $\alpha \in \mathbb{R}^N$ and $\delta > 0$ find $u_\delta \in U_\sigma$ minimising $J_\delta$ over $U_\sigma$.
\end{problem}
\begin{proposition} \label{wellPosed_pointConPenMethod}
For each $\delta>0$ there exists a unique solution $u_\delta$ to Problem \ref{pointConPenMethod}.
\end{proposition}
\begin{proof}
Define a bilinear form $a_\sigma^\delta:U_\sigma\times U_\sigma \rightarrow \mathbb{R}$ by
\[
a_\sigma^\delta (u,v):= a_\sigma(u,v) + \frac{1}{\delta}\sum_{i=1}^N u(X_i)v(X_i).
\]  
Notice $a_\sigma^\delta$ is bounded, symmetric and positive semi-definite, hence weak lower semi-continuous. In fact it is also coercive, we prove this by contradiction. Assume $a_\sigma^\delta$ is not coercive over $U_\sigma$, then there exists a sequence $u_n \in U_\sigma$ such that
\[
\|u_n\|_{H^2(\Gamma)} = 1 \text{ and } a_\sigma^\delta(u_n,u_n) < \frac{1}{n} \text{ for all } n \geq 1.
\]
Then we may find a subsequence $u_{n'} \rightharpoonup u$ for some $u \in U_\sigma$, it follows
\[
0 \leq a_\sigma^\delta(u,u) \leq \lim_{n' \rightarrow \infty} a_\sigma^\delta(u_{n'},u_{n'}) = 0.
\]
Thus $u \in U_\sigma \cap Ker(a_\sigma^\delta) = U_\sigma \cap (H^2_{X,0}(\Gamma) \cap Ker(a))=\left\{0\right\}$. Now for each $n'$, $u_{n'}$ may be expressed uniquely in the form
\[
u_{n'} = p_{n'} + q_{n'} \text{ with } p_{n'} \in Ker(a)^\perp \text{ and } q_{n'} \in Ker(a).
\]
The bilinear form $a_\sigma$ is coercive over $Ker(a)^\perp$, hence
\[
C\|p_{n'}\|_{H^2(\Gamma)}^2 \leq a_\sigma(p_{n'},p_{n'}) = a_\sigma(u_{n'},u_{n'}) \leq a_\sigma^\delta(u_{n'},u_{n'}) \rightarrow 0,
\]
thus $p_{n'}\rightarrow 0$. It follows that $q_{n'}\rightharpoonup 0$ and thus $q_{n'} \rightarrow 0$ as $Ker(a)$ is finite dimensional. We have reached a contradiction as now it holds that 
\[
1 = \|u_{n'}\|_{H^2(\Gamma)}^2 = \|p_{n'}\|_{H^2(\Gamma)}^2 + \|q_{n'}\|_{H^2(\Gamma)}^2 \rightarrow 0.  
\]
Hence there exists $\gamma > 0$ such that 
\[
\gamma \|u\|_{H^2(\Gamma)}^2 \leq a_\sigma^\delta(u,u) \quad \forall u \in U_\sigma. 
\]
Now write $J_\delta$ in the form
\[
J_\delta(u) = \frac{1}{2}a_\sigma^\delta(u,u) - \frac{1}{\delta}\sum_{i=1}^N \alpha_i u(X_i) + \frac{1}{2\delta}\sum_{i=1}^N \alpha_i^2.
\]
The existence of a unique solution to Problem \ref{pointConPenMethod} is then a consequence of the Lax-Milgram theorem.
\end{proof}

\noindent We now state the limit problem which solutions to Problem \ref{pointConPenMethod}, $u_\delta$, converge to as $\delta \downarrow 0$, first we introduce the notion of prescribed point values. 

Prescribed point values at distinct locations $X=(X_i) \in \Gamma^N$ will be represented by the constraints 
\begin{equation} \label{pointConstraints}
F_X(u)=\alpha
\end{equation}
with given $\alpha \in \mathbb{R}^N$ and $F_X$ defined by
\begin{equation}\label{deltaFunctional}
F_X(u) = (u(X_1),...,u(X_N))\in \mathbb{R}^N.
\end{equation}
Note that this is well defined as the map $u \mapsto u(X_i) \in H^2(\Gamma)'$ due to the continuous embedding $H^2(\Gamma) \subset C(\Gamma)$. With this concept of prescribed point values we can state the limit problem.

\begin{problem}[Point value constraints]\ \\ \label{prob:hcon}
Find $u \in (U_\sigma)_{X,\alpha}$ satisfying the following equivalent conditions.
\begin{enumerate}
\item $u \in (U_\sigma)_{X,\alpha}$ minimises $u \mapsto \frac{1}{2} a_\sigma(u,u)$ over $(U_\sigma)_{X,\alpha}$.
\item $u \in (U_\sigma)_{X,\alpha}$ is such that $a_\sigma(u,v)=0$ for all $v \in (U_\sigma)_{X,0}$.  
\end{enumerate}
\end{problem}

Existence and uniqueness of a solution to this problem follows from the fact that $a_\sigma(\cdot,\cdot)$ is coercive over $(U_\sigma)_{X,0}$, which is deduced from the observation $a_\sigma^\delta(v,v) = a_\sigma(v,v)$ for $v \in (U_\sigma)_{X,0}$, as $a_\sigma^\delta$ is coercive over $U_\sigma$, shown in the proof of Proposition \ref{wellPosed_pointConPenMethod}. 
We now show convergence for the penalty method.

\begin{proposition}
For $\delta > 0$ let $u_\delta$ denote the unique solution of Problem \ref{pointConPenMethod} and $u$ the unique solution of Problem \ref{prob:hcon}, then $u_\delta \rightarrow u$ in $H^2(\Gamma)$-norm as $\delta \downarrow 0$.
\end{proposition}
\noindent For a proof see \cite{EllGraHobKorWol15}, Proposition 9.3. 

A natural question to consider is minimising over the constraint locations $X$ as well as the displacement $u$, analogous to Problem \ref{prob:varying_forces}. The general theory in \cite[Proposition 9.4]{EllGraHobKorWol15}, is used there to establish a minimum for this type of problem, when the constraint locations $X$ are allowed to vary in the planar case. This theory cannot be applied here however as the underlying Hilbert space $U_\sigma$ depends upon the locations $X$.

\section{Numerical studies}  %%% SECTION  %%%%
\label{Section_numerics}

\subsection{Discretization}

In this section we present some preliminary illustrative numerical results concerning problems formulated in Section \ref{Section_point_forces} concerning the Willmore functional. Our numerical studies are performed using surface finite elements, \cite{DziEll13}.
The underlying partial differential equations are of fourth order. In order to avoid the use of $H^2$ conforming surface finite elements we use second order splitting to obtain two coupled second order surface equations which can be approximated by continuous  piece-wise linear surface finite elements on triangulated surfaces. The analysis of these schemes will be considered in a later work.

We now assume that the undeformed surface $\Gamma_0$ is approximated by a polyhedral hypersurface
$$
	\Gamma_h = \bigcup_{T \in \mathcal{T}_h} T,
$$ 
where $\mathcal{T}_h$ denotes the set of two-dimensional simplices in $\mathbb{R}^{3}$
which are supposed to form an admissible triangulation. Recall that our problems are posed on either a sphere or a Clifford torus. 
The approach is also applicable to similar PDEs on other closed surfaces. We assume that $\Gamma_h$ is
contained in a strip $\mathcal{N}_\delta$ of width $\delta > 0$ on which the
decomposition 
$$
	x = p + d(x) \nu(p), \quad p \in \Gamma 
$$ 
is unique for all $x \in \mathcal{N}_\delta$. Here, $d(x)$ denotes the oriented distance function to $\Gamma$,
see Section 2.2 in \cite{DecDziEll05}. This defines a map $x \mapsto p(x)$ from $\mathcal{N}_\delta$ onto 
$\Gamma$. We here assume that the restriction $p_{|\Gamma_h}$ 
of this map onto the polyhedral hypersurface $\Gamma_h$
is a bijective map between $\Gamma_h$ and $\Gamma$. In addition, the vertices of the simplices 
$T \in \mathcal{T}_h$ are supposed to sit on $\Gamma$. The generation of these triangulations for the sphere and torus is rather standard, see for example \cite{DziEll13}. In Figure \ref{fig:exampleMeshes} we show typical triangulations.

\begin{figure}
\centering
\begin{subfigure}[b]{0.45\textwidth}
\includegraphics[width=\textwidth]{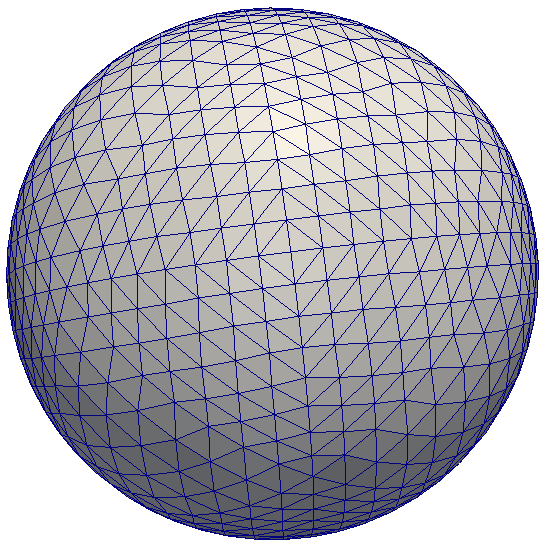} 
\caption{Example triangulation for a sphere} \label{fig:exampleMesh_sphere}
\end{subfigure}
\hspace{0.05\textwidth}
\begin{subfigure}[b]{0.45\textwidth}
\includegraphics[width=\textwidth]{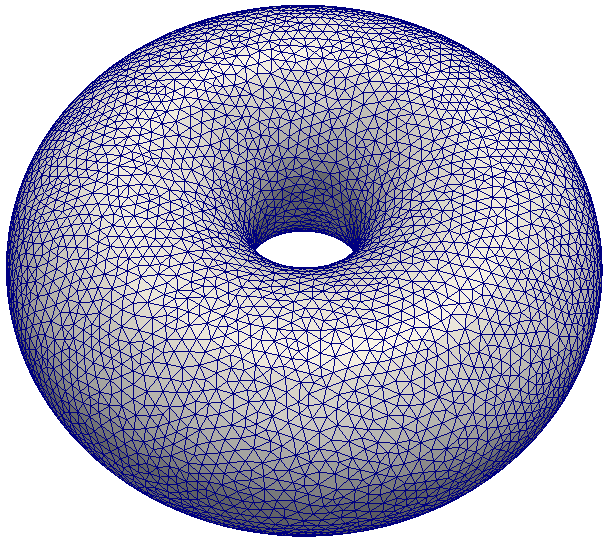} 
\caption{Example triangulation for a torus} \label{fig:exampleMesh_torus}
\end{subfigure}
\caption{Example triangulations of surfaces.}
\label{fig:exampleMeshes}
\end{figure}

The piecewise linear Lagrange finite element space on $\Gamma_h$ is
$$
	\mathcal{S}_h := \left\{ \chi \in C(\Gamma_h) \;|\; \chi_{T} \in \mathbb{P}_1(T)\;
	\forall T \in \mathcal{T}_h \right\},
$$ 
where $\mathbb{P}_1(T)$ denotes the set of polynomials of degree $1$ on $T$.
The Lagrange basis functions $\varphi_i$ of this space are uniquely determined by their values
at the so-called Lagrange nodes $q_j$, that is $\varphi_i(q_j) = \delta_{ij}$. 
The associated Lagrange interpolations for a continuous function $f$ on $\Gamma_h$ are defined by
$$
	I_h^r f := \sum_{i} f(q_i) \varphi_i.
$$

\subsubsection{Point forces on a sphere}
We first look to solve Problem \ref{prob:fixed_forces} numerically, using a second order  splitting method and enforcing the constraint $u \in V= Ker(a)^\perp$ via the addition of a new bilinear form and an adjustment of the right hand side. For the numerical method we will express the condition $u \in V$ via $L^2(\Gamma)$-orthogonality. That is, for each basis function of $Ker(a)$, $f_i$, set 
\[
g_i:= (\DeltaG^2 - \DeltaG + 1)f_i \in C^\infty(\Gamma).
\]
We can then characterise $V$ in terms of $L^2(\Gamma)$-orthogonality with the $g_i$, for $1\leq i \leq M:=dim(Ker(a))$,
\[
V = \left\{ v \in H^2(\Gamma) ~|~ (v,g_i)_{L^2(\Gamma)}=0 ~\forall 1 \leq i \leq M \right\}.
\]
Using the basis functions given in Lemma \ref{lem:Kera}, for a sphere the $g_i$ are given by
\[
\left\{ 1, \left(\frac{4}{R^4} + \frac{2}{R^2} + 1 \right)\nu_i ~\bigg|~ i=1,2,3 \right\},
\] 
in our applications we will neglect the multiplicative constant appearing in front of the $\nu_i$.
For a Clifford torus the formulae are somewhat lengthier.
We may assume, without loss of generality, that the basis functions $g_i$ form an $L^2(\Gamma)$-orthogonal set.

In this section we will consider the case $\Gamma=S(0,R)$, a sphere of radius $R$, however the Clifford torus problem can be treated in a similar manner. To formulate the splitting method we first consider the Euler Lagrange equation associated to minimising the point forces energy in Problem \ref{prob:fixed_forces}. This energy minimisation is equivalent to the following variational problem.
\begin{problem}
Find $u \in V$ such that 
\[
a_\sigma(u,v) = \sum_{i=1}^N \beta_i v(X_i) \quad \forall v \in V.
\]
\end{problem} \label{prob:fixedForcesVariational}
This variational problem is posed over $V$, a subspace of $H^2(\Gamma)$. For ease of implementation, we wish to solve a problem posed over the full space $H^2(\Gamma)$. To formulate such a problem we introduce Lagrange multipliers, the resulting variational problem is shown below.
\begin{problem} \label{prob:fixedForcesLagrange}
Find $(u,\lambda) \in H^2(\Gamma) \times \mathbb{R}^4$ such that
\begin{align*}
a_\sigma(u,v) &= \sum_{k=1}^N \beta_k v(X_k) - \lambda_0 \int_\Gamma v \;do - \sum_{i=1}^3 \lambda_i\int_\Gamma v\nu_i \; do    \quad \forall v \in H^2(\Gamma), \\
\int_\Gamma u \;do &= \int_\Gamma u \nu_i \; do = 0 \quad \text{for } i=1,2,3.
\end{align*}
\end{problem} 
Here we can easily determine the Lagrange multipliers. Testing the first equation with a component of the normal, $\nu_i$, or the constant function $1$ we obtain zero on the left hand side, using $\int_\Gamma u \;do = 0$ for the latter case. The right hand side must likewise vanish, determining the Lagrange multipliers
\[
\lambda_0 = \frac{1}{4\pi R^2} \sum_{k=1}^N \beta_k \quad \text{and} \quad \lambda_i = \frac{3}{4\pi R^2}\sum_{k=1}^N \beta_k \nu_i(X_k) \quad \text{for} \quad i=1,2,3.
\]  
Now we have produced a problem for $u$ which is posed over the full space $H^2(\Gamma)$. We could now discretise and solve this problem using $H^2(\Gamma)$-conforming finite elements. This approach will not be taken here however, instead we will formulate an equivalent problem which can be solved using lower order finite elements. To do so we will split this fourth order problem into a system of two second order equations. Such a splitting is best motivated by considering the strong form of Problem \ref{prob:fixedForcesLagrange}, produced by integrating the variational problem by parts.
\begin{equation} \label{eq:sphereEulerLagrange}
\begin{split}
\kappa \DeltaG^2 u - \left(\sigma - \frac{2\kappa}{R^2} \right) \DeltaG u - \frac{2\sigma}{R^2}u &= \tilde{\delta}_X, \\
\int_\Gamma u \;do = \int_\Gamma u \nu_i  \; do &= 0 \quad \text{for } i=1,2,3.
\end{split}
\end{equation}
This equation is meant only in the sense of distributions, here $\tilde{\delta}_X$ denotes
\[
\tilde{\delta}_X := \sum_{k=1}^N \beta_k \left( \delta_{X_k} - \frac{1}{4\pi R^2} - \sum_{i=1}^3\frac{3}{4\pi R^2} \nu_i(X_k) \nu_i \right)
\]
where $\delta_{X_k}$ is the Dirac delta distribution $\delta_{X_k}: v \mapsto v(X_k)$. The PDE can also be given meaning in the sense that both sides lie in $H^2(\Gamma)'$, this leads to the variational problem above.

To solve \eqref{eq:sphereEulerLagrange} numerically we will use a splitting method to formulate this fourth order problem as a pair of second order equations. The splitting occurs by introducing the new variable $w$, which satisfies 
\[
w = -\DeltaG u - \frac{2}{R^2}u,
\]
it is immediate that $w$ must satisfy the constraints
\[
\int_\Gamma w \;do = \int_\Gamma w \nu_i \; do = 0 \quad \text{for } i=1,2,3.
\]
Using this splitting we are left with a decoupled system, given by
\begin{align*}
-\kappa\DeltaG w + \sigma w &= \tilde{\delta}_X \text{ with constraints } \int_\Gamma w \;do = \int_\Gamma w \nu_i \; do = 0 \quad \text{for } i=1,2,3, \\
-\DeltaG u - \frac{2}{R^2} u &= w \text{ with constraints } \int_\Gamma u \;do = \int_\Gamma u \nu_i \; do = 0 \quad \text{for } i=1,2,3.
\end{align*}
We wish to find a weak formulation for this decoupled system to base our numerical method around. In addition, the implementation of the numerical method is much more straightforward if we can make the constraints a property of the equations we solve. We thus make a modification to the decoupled system, the modified system reads
\begin{align*}
-\kappa\DeltaG w + \sigma w +\chi_\sigma \int_\Gamma w &= \tilde{\delta}_X, \\
-\DeltaG u -\frac{2}{R^2} u + \tau \sum_{i=1}^4 \left(\int_\Gamma ug_i \right) g_i &= w,
\end{align*}
where $\tau >(1/2\pi R^4)$ and $\chi_\sigma = 1$ if $\sigma =0$ and is zero otherwise. As $u,w \in V$ we have actually added zero to both equations but in doing so we ensure that any solution to this modified system must satisfy the required constraints. We now give an appropriate weak formulation which will be discretised to produce the finite element method.

\begin{problem} \label{prob:forcesSplittingMethod}
Fix $X \in \Gamma^N$, $p \in (1,2)$ and $q \in (2,\infty)$ such that $1/p +1/q =1$. Find $(u,w) \in H^1(\Gamma) \times W^{1,p}(\Gamma)$ such that 
\begin{align*}
\int_\Gamma \kappa \nablaG w \cdot \nablaG v +\sigma w v \;do + \chi_\sigma (w,1)_{L^2(\Gamma)}(v,1)_{L^2(\Gamma)}  &=  \langle \tilde{\delta}_X , v \rangle  \qquad \forall v \in W^{1,q}(\Gamma), \\
\int_\Gamma \nablaG u \cdot\nablaG v -\frac{2}{R^2} u v \;do +  \tau \sum_{i=1}^{4} (u,g_i)_{L^2(\Gamma)}(v,g_i)_{L^2(\Gamma)} &= \int_\Gamma w v \;do \quad \forall v \in H^1(\Gamma).
\end{align*}
\end{problem}

This formulation is well posed and is equivalent to Problem \ref{prob:fixed_forces} in the sense that the solution is given by $(u,-\DeltaG u -(2/R^2) u)$ where $u$ is the solution to Problem \ref{prob:fixed_forces}. This result and the numerical analysis of the method we now introduce will be the subject of a later paper.
We discretise the system using $P^1$ finite elements, the resulting equations for the discrete system are as follows.   

\begin{problem} \label{prob:discretePointForces}
Find $\mathbf{u},\mathbf{w} \in \mathbb{R}^{N_h}$ such that
\begin{align*}
\left(\kappa S_h  + \sigma M_h +\chi_\sigma A_h \right) \mathbf{w} &= F_h \\
\left(S_h -\frac{2}{R^2} M_h +\tau B_h \right) \mathbf{u} &= M_h \mathbf{w} 
\end{align*}
\end{problem}

\noindent Here $M_h$ and $S_h$ are the usual mass and stiffness matrices respectively,
\[
(M_h)_{ij} := \int_{\Gamma_h} \phi_i \phi_j \;do_h \quad \text{and} \quad (S_h)_{ij} := \int_{\Gamma_h} \nabla_{\Gamma_h} \phi_i \cdot \nabla_{\Gamma_h} \phi_j \;do_h. 
\]
The matrix $A_h$ is given by
\[
(A_h)_{ij}  := (\phi_i,1)_{L^2(\Gamma_h)}(\phi_j,1)_{L^2(\Gamma_h)}. 
\]
The matrix $B_h$ is given by
\[
(B_h)_{ij}  := (\phi_i,1)_{L^2(\Gamma_h)}(\phi_j,1)_{L^2(\Gamma_h)} + \sum_{k=1}^3  (\phi_i,\nu_k\circ p)_{L^2(\Gamma_h)}(\phi_j,\nu_k\circ p)_{L^2(\Gamma_h)}. 
\]
The right hand side in the first equation results from approximating the linear functional $\tilde{\delta}_X$ and is given by
\[
(F_h)_i := \sum_{k=1}^N \beta_k\left(\phi_i\circ p_{|\Gamma_h}^{-1}(X_k) -  \frac{1}{4\pi R^2}(\phi_i,1)_{L^2(\Gamma_h)} - \frac{3}{4\pi R^2}\sum_{r=1}^3 \nu_r(X_k) (\phi_i,\nu_r\circ p)_{L^2(\Gamma_h)} \right).
\]
The convergence of this finite element method will not be addressed here, instead we will introduce a similar method for the point constraints problem before producing some illustrative examples.

\subsubsection{Point constraints for a Clifford torus}
We also solve Problem \ref{pointConPenMethod} numerically, enforcing the constraint $u \in U= (H^2_{X,0} \cap Ker(a))^\perp$ via a penalty method. As in the previous algorithm we will express the condition $u \in U$ via $L^2(\Gamma)$-orthogonality. That is, for each basis function of $H^2_{X,0} \cap Ker(a)$, $f_i$, set 
\[
g_i:= (\DeltaG^2 - \DeltaG + 1)f_i \in C^\infty(\Gamma).
\]
We can then characterise $U$ in terms of $L^2(\Gamma)$-orthogonality with the $g_i$, for $1\leq i \leq L:=dim(H^2_{X,0} \cap Ker(a))$
\[
U = \left\{ v \in H^2(\Gamma) ~|~ (v,g_i)_{L^2(\Gamma)}=0 ~\forall 1 \leq i \leq L \right\}
\]
The resulting minimisation problem, which we will discretise, is given as follows.
\begin{problem}

Given $X \in \Gamma^N$ and $\delta,\rho >0$ find $u_{\delta,\rho} \in H^2(\Gamma)$ minimising $\mathcal{E}_{\delta,\rho}(\cdot,X)$ over $H^2(\Gamma)$, where
\[
\mathcal{E}_{\delta,\rho} (u,X):= \frac{1}{2}a(u,u) + \frac{1}{2\delta} \sum_{i=1}^N (u(X_i)-\alpha_i)^2 +\frac{1}{2\rho} \sum_{i=1}^{L} (u,g_i)_{L^2(\Gamma)}^2.
\]
\end{problem}

Existence and uniqueness of a solution is a consequence of the Lax-Milgram theorem and by standard techniques for penalty methods one can show $\|u_{\delta,\rho} -u_\delta \|_{H^2(\Gamma)} \rightarrow 0$ as $\rho \rightarrow 0$, where $u_\delta$ is the solution to Problem \ref{pointConPenMethod}. Here we will consider the case $\Gamma=T(1,\sqrt{2})$, a Clifford torus.

We discretise and solve the problem using the splitting $w=-\DeltaG u + u$. The resulting equations are as follows.
\begin{problem}
Find $\mathbf{u},\mathbf{w} \in \mathbb{R}^{N_h}$ such that
\begin{align*}
\left(
\begin{array}{cc}
T_h + \frac{1}{\delta}C_h + \frac{1}{\rho}B_h & S_h + M_h \\
S_h + M_h & -M_h
\end{array}
\right) 
\left(
\begin{array}{c}
\mathbf{u} \\
\mathbf{w}
\end{array}
\right) = 
\left( 
\begin{array}{c}
\tfrac{1}{\delta} \tilde{F}_h \\
0
\end{array}
\right)
\end{align*}
\end{problem}

\noindent Here $M_h$ and $S_h$ are the usual mass and stiffness matrices respectively. The matrix $T_h$ is induced by the inner product
\begin{align*}
	t(u, v)
	=&  \int_{\Gamma} \nabla_{\Gamma} u \cdot \left( \left[\frac{3}{2}H^2 - 2|\mathcal{H}|^2 -2 \right]\unit - 2 H \mathcal{H} \right) \nabla_{\Gamma} v 
	\\  
	 &+ uv \left( - \frac{3}{2} H^2 | \mathcal{H} |^2 + 2 ( \nabla_{\Gamma} \nabla_{\Gamma} H) : \mathcal{H} + |\nabla_{\Gamma} H|^2 + 2 H Tr(\mathcal{H}^3) +\DeltaG |\mathcal{H}|^2 +|\mathcal{H}|^4 -1\right) \; do.
\end{align*} 
This term in the equation results from the fact 
\[
a(u,v) = \int_\Gamma (-\DeltaG u + u)(-\DeltaG v + v) \;do + t(u,v),
\]
as when we carry out the splitting $w=-\DeltaG u + u$ this becomes
\[
a(u,v) = \int_\Gamma \nablaG w \cdot \nablaG v + wv \;do + t(u,v).
\] 
For the discretised equations this contributes $S_h + M_h$ to the upper right block of the system matrix and the remainder $T_h$ to the upper left block.    

The matrix $B_h$ results from the penalty terms for the elements of $Ker(a)$,
\[
(B_h)_{ij}  := \sum_{k=1}^L  (\phi^1_i,g_k\circ p)_{L^2(\Gamma_h)}(\phi^1_j,g_k\circ p)_{L^2(\Gamma_h)}. 
\]
The matrix $C_h$ results from the penalty terms for the point constraints,
\[
(C_h)_{ij}:= \sum_{k=1}^N \phi^1_i\circ p_{|\Gamma_h}^{-1}(X_k)\phi^1_j\circ p_{|\Gamma_h}^{-1}(X_k).
\]
The first block of right hand side vector is given by
\[
(\tilde{F}_h)_i := \sum_{k=1}^N \alpha_k\phi^1_i\circ p_{|\Gamma_h}^{-1}(X_k).
\]
The numerical analysis of these methods will be the subject of a later paper. 

\begin{remark}
The point constraints problem may also be approached by adjusting the linear functionals $u \mapsto u(X_k)$ as was done for point forces in \eqref{eq:sphereEulerLagrange}. Similarly the point forces problem may be approached by a penalty method by penalising each of the integrals $(u,g_i)_{L^2(\Gamma)}$.   
\end{remark}

\subsection{Numerical results}
\subsubsection{Point forces on a sphere}

As in the flat case \cite{EllGraHobKorWol15} we investigate the membrane mediated interactions between point forces. To do so we solve the discrete problem, Problem \ref{prob:discretePointForces}, with $R=1$, $N=2$, $X_1=(0,0,1)$ and and $X_2=(\sin(\theta),0,\cos(\theta))$, varying $\theta \in [0,\pi]$. We take $\beta_1=5$ and consider each of the cases $\beta = \pm 5$. As in the flat case, we fix $\kappa=1$ but use varying values for the surface tension $\sigma$ to explore how the ratio $\kappa/\sigma$ affects the interactions. 

Figure \ref{fig:Energy_Plot_same} plots the energy of the discrete solution as a function of $\theta$ for point forces with the same sign, $\beta_1=\beta_2=5$. At small separations we observe a similar attractive interaction as was observed in the flat case \cite{EllGraHobKorWol15}. This agrees with the attractive interaction between filopodia discussed in biophysics literature \cite{AtiWirSun06,IsaManKacYocGov13}. However there is a critical separation angle $\theta_c$ beyond which the interaction is repulsive. This repulsion at larger separations cannot be observed in the flat case as it occurs precisely when the membrane is no longer well approximated by a planar graph. The global minimum is at $\theta = 0$, corresponding to the two forces clustering to the same point as was observed in the flat case and proven by the general theory. There is also a local minimum at $\theta = \pi$, corresponding to the forces being located at opposite poles.

\begin{figure}[h!]
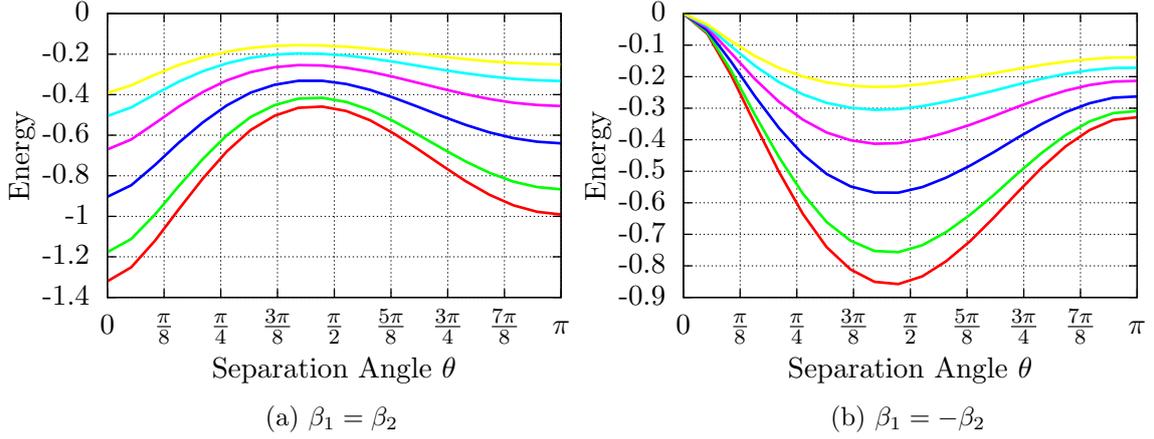

\centering
\hspace{1cm}
\begin{subfigure}[b]{0.45\textwidth}
\hspace{0.5cm}
\input{gnuplot_file-same.tex} 
\vspace{1cm}
\caption{$\beta_1=\beta_2$} \label{fig:Energy_Plot_same}
\end{subfigure}
\hspace{0.01\textwidth}
\begin{subfigure}[b]{0.45\textwidth}
\hspace{0.5cm}
\input{gnuplot_file-opp.tex} 
\vspace{1cm}
\caption{$\beta_1=-\beta_2$} \label{fig:Energy_Plot_opp}
\end{subfigure}
\caption{Energy plots for forces with identical and opposite orientations, varying $\sigma$ from $0$ to $25$ (bottom to top).}
\end{figure} 

Figure \ref{fig:Energy_Plot_opp} plots the energy of the discrete solution as a function of $\theta$ for point forces with the opposite sign, $\beta_1=-\beta_2=5$. At small separations we observe a similar repulsive interaction as was observed in the flat case. As for the previous example, the interaction changes at the critical angle $\theta_c$, in this case becoming attractive. This leads to the global minimum occurring at $\theta = \theta_c$.  

The existence of this critical angle and its dependence on $\sigma$ can be seen by studying $G_h$, the solution for $N=1$, $X_1=(0,0,1)$ and $\beta=1$. The $\theta$-dependent part of the discrete energy for the two examples above may be written as
\[
E_h(\theta) = -\beta_1\beta_2 G_h(X_2(\theta)).
\] 
Thus when $\beta_1$ and $\beta_2$ have the same sign, the energy is least when $G_h > 0$ and when they have opposite sign the energy is least when $G_h < 0$. Moreover the critical angle $\theta_c$ is precisely the angle which minimises $G_h(X(\theta))$. Figure \ref{fig:sol_sigma=0} plots $G_h$ for $\sigma=0$ and Figure \ref{fig:soln_sigma=25} for $\sigma=25$. The red regions are areas where $G_h$ is positive and the blue where it is negative. The values are plotted onto a surface representative of the deformed surface $G_h$ produces in each case. So that the deformations are visible, they have not been scaled by $\varepsilon$ for these plots.   Also overlayed on each figure is the line along which the minimum occurs, that is the line $\theta=\theta_c$. For $\sigma = 0$ we have $\theta_c \approx 83^\circ$ and for $\sigma =25$ we have $\theta_c \approx 77^\circ$. One observes that as $\sigma$ increases the effect of the force becomes more localised, shrinking the positive, red region and decreasing the value of $\theta_c$.    

\begin{figure}
\centering
\begin{subfigure}[b]{0.45\textwidth}
\includegraphics[width=\textwidth]{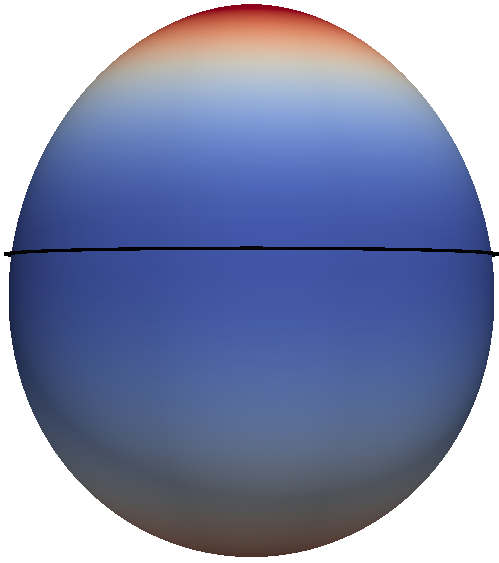} 
\caption{$\sigma=0$} \label{fig:sol_sigma=0}
\end{subfigure}
\hspace{0.05\textwidth}
\begin{subfigure}[b]{0.45\textwidth}
\includegraphics[width=\textwidth]{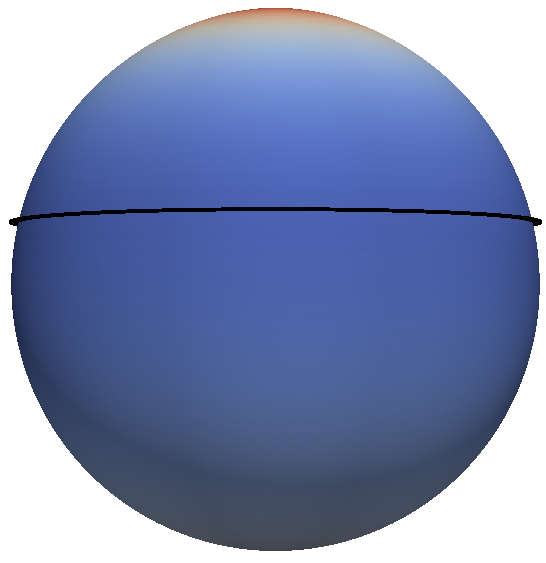} 
\caption{$\sigma=25$} \label{fig:soln_sigma=25}
\end{subfigure}
\caption{Plot of $G_h$ values on $\Gamma_h$ for varying $\sigma$.}
\end{figure}   

\subsubsection{Point constraints for a Clifford torus}
For the second algorithm we will simply provide some illustrative examples of numerical solutions.
Figure \ref{fig:example1_torus} shows the deformed surface produced when minimising the linearised energy under the point constraints $u(X_k)=\alpha_i$ for $k=1,2,3$ with 
\begin{align*}
X_k &= ((\sqrt{2} + 1)\cos((2+k)\pi/4), (\sqrt{2} + 1)\sin((2+k)\pi/4), 0 ), \\
\alpha &= (-0.5,1,-0.5).
\end{align*}
Figure \ref{fig:example2_torus} shows the deformed surface produced when minimising the linearised energy under the point constraints $u(X_k)=\alpha_i$ for $k=1,2,3$ with 
\begin{align*}
X_k &= (-(\sqrt{2} + \cos(2k\pi/3)), 0, \sin(2k\pi/3) ), \\
\alpha &= (-0.5,-0.5,1).
\end{align*}
In both cases there are deformations away from the point constraint locations. As for the point forces on a sphere, this will give rise to longer distance interactions that are not witnessed when the undeformed surface is planar. Note that the figures show the deformed surface $\Gamma_\varepsilon$, here we have chosen $\varepsilon=0.2$. In reality $\varepsilon$ is a much smaller parameter but using a relatively large value for $\varepsilon$ in these plots means the deformations are visible.   

\begin{figure}
\centering
\begin{subfigure}[b]{0.45\textwidth}
\includegraphics[width=\textwidth]{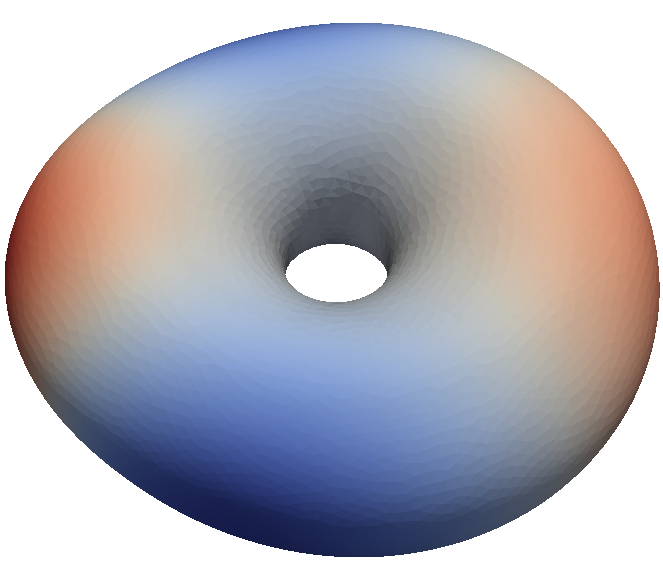} 
\caption{Constraints around outer circle} \label{fig:example1_torus}
\end{subfigure}
\hspace{0.05\textwidth}
\begin{subfigure}[b]{0.45\textwidth}
\includegraphics[width=\textwidth]{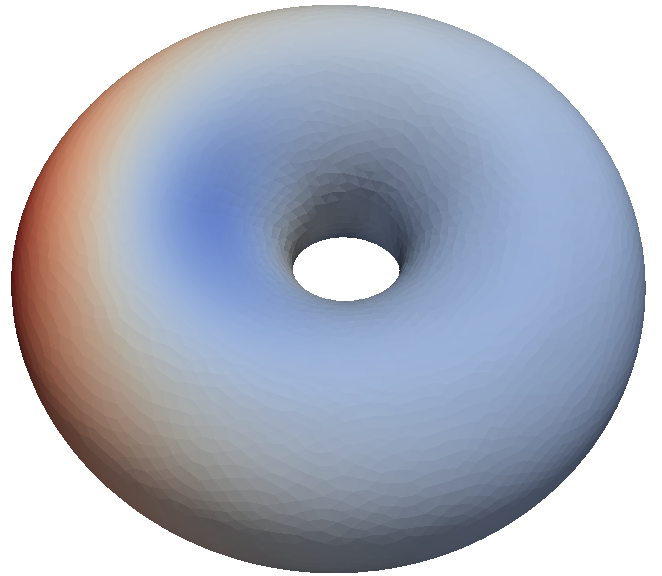} 
\caption{Constraints around inner cricle} \label{fig:example2_torus}
\end{subfigure}
\caption{Examples of deformed Clifford tori subject to point constraints.}
\end{figure}

\newpage
\renewcommand\thesection{\Alph{section}}
\setcounter{section}{0}

\section{Appendix}
\subsection*{Derivation of the second variation of the Willmore functional}
For the sake of completeness we present the derivation of the second variation, see also \cite{Gla11}.
We will not integrate by parts in the formulas below -- unless otherwise stated. 
This means that our results can be more readily adapted to surfaces with boundary.
This might be useful for studying biomembranes with finite-size inclusions. 
The following calculations are valid for $n$-dimensional hypersurfaces $\Gamma \subset \mathbb{R}^{n+1}$.
We begin by calculating the required material derivatives, starting with the unit normal.
\begin{lemma} \label{materialderiv_normal}
Suppose $\Gamma \subset \mathbb{R}^{n+1}$ is a parametrised $n$-dimensional hypersurface, $\Gamma = \left\{ X(\theta) \;|\; \theta \in \Omega \right\}$, let $u\in C^1(\Gamma)$ and define $\Gamma_\mu:=\left\{ X^\mu(\theta):=X(\theta) + \mu \tilde{u}(\theta)\tilde{\nu}(\theta) \;|\; \theta \in \Omega \right\}$ where $\tilde{u}(\theta):=u(X(\theta))$ and $\tilde{\nu}:=\nu(X(\theta))$. Then the material derivative of the normal is given by
\[
\dot{\partial}\nu^\mu = -\nabla_\Gamma u. 
\]
\end{lemma}  
\begin{proof}
We have
\[
0 = \dot{\partial} \left(|\nu^\mu|^2 \right) = 2\nu \cdot \dot{\partial}(\nu^\mu)  
\]
thus $\dot{\partial}(\nu^\mu) \in \nu^\bot$ so let $\dot{\partial}(\nu^\mu) \circ X = \sum_{i=1}^n \alpha_i X_{\theta_i}$. It then follows
\[
\sum_{i=1}^n \alpha_i g_{ij} = \left(\dot{\partial}(\nu^\mu)\circ X\right)\cdot X_{\theta_j} = -\left( \dot{\partial}(X^\mu_{\theta_j}) \right) \cdot \tilde{\nu} = -\tilde{\nu} \cdot (\tilde{u}\tilde{\nu})_{\theta_j} = -\tilde{u}_{\theta_j}.
\]
We may then conclude
\[
\dot{\partial}(\nu^\mu) \circ X = \sum_{i=1}^n \alpha_i X_{\theta_i} = \sum_{i,j,k=1}^n \alpha_i g_{ik} g^{jk} X_{\theta_j} = -\sum_{j,k=1}^n  g^{jk} \tilde{u}_{\theta_k} X_{\theta_j} = -\nablaG u \circ X.  
\]
\end{proof}
Now we will calculate the material derivative for the entries of the inverse of the first fundamental form.

\begin{lemma} Denote the entries of the inverse of the first fundamental form $G^\mu$ by $g^{\mu ij}$.
\[
\dot{\partial} g^{\mu ij} = - \sum_{k,l=1}^n \tilde{u}g^{il}g^{kj}(X_{\theta_k}\cdot \tilde{\nu}_{\theta_l} + X_{\theta_l}\cdot \tilde{\nu}_{\theta_k}) 
\]
\end{lemma}
\begin{proof}
\begin{align*}
\dot{\partial} g^{\mu ij} &= \dot{\partial}\left( \sum_{k,l=1}^n g^{\mu kj}g^{\mu li}g^\mu_{lk} \right) \\
&= \sum_{k=1}^n \left( \dot{\partial} g^{\mu kj} \right) \delta_{k}^i + \sum_{l=1}^n\left( \dot{\partial} g^{\mu li} \right) \delta_{l}^j + \sum_{k,l=1}^ng^{kj}g^{li} \left( \dot{\partial} g^\mu_{lk} \right)  \\
&= 2\dot{\partial} g^{\mu ij} + \sum_{k,l=1}^n \tilde{u} g^{kj}g^{li}(X_{\theta_k}\cdot \tilde{\nu}_{\theta_l} + X_{\theta_l}\cdot \tilde{\nu}_{\theta_k})    
\end{align*}
\end{proof}

Next, we will derive the material derivative of the tangential gradient.
\begin{lemma} \label{materialderiv_deriv}
Suppose $\Gamma,\Gamma_\mu$ are as in Lemma \ref{materialderiv_normal} and let $f^\mu :\Gamma_\mu \rightarrow \mathbb{R}$ then
\[
\dot{\partial}\left( {\nabla_{\Gamma^\mu}} f^\mu \right) = -u\mathcal{H}\nablaG f + (\nablaG f \cdot \nablaG u)\nu + \nablaG\left( \dot{\partial} f^\mu \right).
\] 
\end{lemma}
\begin{proof}
\begin{align*}
&\dot{\partial}\left( \nabla_{\Gamma^\mu} f^\mu \right) \circ X = \dot{\partial}\left( \sum_{i,j=1}^n g^{\mu ij} \tilde{f}^\mu_{\theta_i} X^\mu_{\theta_j} \right) \\
&= -\sum_{i,j,k,l=1}^n \tilde{u}g^{kj}g^{li}(X_{\theta_k}\cdot \tilde{\nu}_{\theta_l} + X_{\theta_l}\cdot \tilde{\nu}_{\theta_k})\tilde{f}_{\theta_i}X_{\theta_j} 
+ \sum_{i,j=1}^n \left( g^{ij}  \tilde{f}_{\theta_i}\dot{\partial} X^\mu_{\theta_j} + g^{ij}X_{\theta_j}\dot{\partial}\left( \tilde{f}^\mu \right)_{\theta_i} \right) \\
&= -\sum_{i,l=1}^n \sum_{\gamma=1}^{n+1} \tilde{u}g^{il}\tilde{f}_{\theta_i} \tilde{\nu}_{\gamma_{\theta_l}} \nablaG X_\gamma  - \tilde{u} (\mathcal{H}\nablaG f) \circ X + \sum_{i,j=1}^n g^{ij}\tilde{f}_{\theta_i} \left(\tilde{u}_{\theta_j}\tilde{\nu} + \tilde{u}\tilde{\nu}_{\theta_j} \right) + \nablaG \left(\dot{\partial} f^\mu\right) \circ X \\
&= \left( - u \mathcal{H}\nablaG f + (\nablaG f \cdot \nablaG u)\nu + \nablaG \left(\dot{\partial} f^\mu\right) \right) \circ X       
\end{align*}
\end{proof}

These three lemmas combined produce the following calculations.
\begin{corollary}
\label{mean_curvature_1var}
The mean curvature $H^\mu = \nabla_{\Gamma^\mu} \cdot \nu^\mu$ satisfies
\[
\dot{\partial} H^\mu = -\DeltaG u -|\mathcal{H}|^2 u.  
\]
The extended Weingarten map $\mathcal{H}^\mu=\nabla_{\Gamma^\mu} \nu^\mu$ satisfies
\[
\dot{\partial} \left( |\mathcal{H}^\mu|^2 \right) = -2uTr\left(\mathcal{H}^3\right) -2 \mathcal{H}:\nablaG \nablaG u.
\]
\end{corollary}

We may now calculate the final material derivative required for the second variation.

\begin{lemma} For the material derivative of $\Delta_{\Gamma^\mu} f^\mu$ we have 
\begin{align*}
\dot{\partial} \left( \Delta_{\Gamma^\mu} f^\mu \right) &= -2u\mathcal{H}:\nablaG \nablaG f -2\mathcal{H}\nablaG f\cdot \nablaG u - u\nablaG f \cdot \nablaG H \\
& \quad + H \nablaG f \cdot \nablaG u + \DeltaG \left(\dot{\partial} f^\mu \right) 
\end{align*}

\end{lemma}
\begin{proof}
Using the commutator rule in (\ref{commutator_rule}), we obtain
\begin{align*}
& \dot{\partial} \left( \Delta_{\Gamma^\mu} f^\mu \right) = \sum_{\alpha=1}^{n+1} -u\nablaG \underline{D}_\alpha f \cdot \nablaG \nu_\alpha + \sum_{\alpha,\beta=1}^{n+1} \nu_\alpha \underline{D}_\beta \underline{D}_\alpha f \underline{D}_\beta u +\sum_{\alpha=1}^{n+1} \underline{D}_\alpha \left( \dot{\partial} \underline{D}^\mu_\alpha f^\mu  \right) \\
&= -u \mathcal{H}:\nablaG \nablaG f + \sum_{\alpha,\beta=1}^{n+1}
\left(
 \nu_\alpha \underline{D}_\alpha \underline{D}_\beta f \underline{D}_\beta u 
 + \sum_{\gamma=1}^{n+1} \nu_\alpha \underline{D}_\beta u
\left(  \nu_\beta \underline{D}_\alpha \nu_\gamma \underline{D}_\gamma f  - \nu_\alpha \underline{D}_\beta \nu_\gamma \underline{D}_\gamma f 
\right) \right)\\
& \quad + \sum_{\alpha,\beta=1}^{n+1} -\underline{D}_\alpha \left( u\underline{D}_\beta f \underline{D}_\beta \nu_\alpha \right) + \underline{D}_\alpha (\nu_\alpha \underline{D}_\beta f \underline{D}_\beta u) +\underline{D}_\alpha\underline{D}_\alpha \left( \dot{\partial} f^\mu \right) \\
&=-2u\mathcal{H}:\nablaG \nablaG f  -2\mathcal{H}\nablaG f\cdot \nablaG u - u\nablaG f \cdot \nablaG H + H \nablaG f \cdot \nablaG u + \DeltaG \left( \dot{\partial} f^\mu \right), 
\end{align*}
where in the last step we have made use of the identity
$$
	\sum_{\alpha=1}^{n+1}\underline{D}_\alpha \underline{D}_\beta \nu_\alpha
	= \underline{D}_\beta H - |\mathcal{H}|^2 \nu_\beta.
$$
\end{proof}

We now use these results to calculate the second variation of the Willmore functional:
\[
W(\Gamma):= \frac{1}{2} \int_\Gamma H^2 \; do
\]
which has first variation, via the transport formula \eqref{transport_formula}, given by
\begin{equation}
\label{Willmore_functional_1var}
	W'(\Gamma)[u\nu] = \int_\Gamma H \dot{\partial} \left( H^\mu \right) +\frac{1}{2}H^3u \; do   =  \int_\Gamma -H \DeltaG u -  H|\mathcal{H}|^2u +\frac{1}{2}H^3u \;do.
\end{equation}

\begin{theorem}
\label{Theorem_2var_Willmore}
The second variation of the Willmore functional $W(\Gamma)$ is given by 
\begin{align*}
& W''(\Gamma)[u\nu, g\nu] = \int_\Gamma (\DeltaG g + |\mathcal{H}|^2g)(\DeltaG u + |\mathcal{H}|^2u) +2H \mathcal{H}:(g\nablaG \nablaG u + u\nablaG \nablaG g )  \\
& \quad +2H \mathcal{H}\nablaG u \cdot \nablaG g + Hg \nablaG u \cdot \nablaG H - H^2 \nablaG u \cdot \nablaG g -\frac{3}{2}H^2 u \DeltaG g - H^2 g \DeltaG u \\
& \quad + \left( 2HTr(\mathcal{H}^3) -\frac{5}{2}H^2|\mathcal{H}|^2 + \frac{1}{2}H^4 \right) gu \;do
\end{align*}
\end{theorem}
\begin{proof}
We use the transport formula again and note that $u$ is extended constantly in the normal direction, in accordance with Remark \ref{remark:firstAndSecondVariaitions}.  
\begin{align*}
W''(\Gamma)[u\nu, g\nu] &= \int_\Gamma -(\mat H^\mu)(\DeltaG u + |\mathcal{H}|^2u) - H\mat(\Delta_{\Gamma^\mu} u^\mu) - H\mat (|\mathcal{H}^\mu|^2)u + \frac{3}{2}H^2 (\mat H^\mu)u \\
& \qquad -H^2g\DeltaG u -H^2|\mathcal{H}|^2gu + \frac{1}{2} H^4 gu \; do
\end{align*}
Using the results of the above lemmas to calculate the required material derivatives produces the result.
\end{proof}
\begin{remark}
\label{Remark_integration_by_parts_for_Willmore}
Using integration by parts on closed surfaces for the term
\begin{align*}
	\int_\Gamma H g \nablaG u \cdot \nablaG H \;do &= \int_\Gamma \nablaG u \cdot \nablaG \left( \frac{1}{2}H^2g \right) -\frac{1}{2}H^2 \nablaG u \cdot \nablaG g \;do\\
	&= \int_\Gamma -\frac{1}{2}H^2g\DeltaG u - \frac{1}{2}H^2 \nablaG u \cdot \nablaG g \;do
\end{align*}
leads to
\begin{align*}
& W''(\Gamma)[u\nu, g\nu] = \int_\Gamma (\DeltaG g + |\mathcal{H}|^2g)(\DeltaG u + |\mathcal{H}|^2u) +2H \mathcal{H}:(g\nablaG \nablaG u + u\nablaG \nablaG g )  \\
& \quad +2H \mathcal{H}\nablaG u \cdot \nablaG g - \frac{3}{2} H^2 \nablaG u \cdot \nablaG g -\frac{3}{2}H^2 ( u \DeltaG g + g \DeltaG u) \\
& \quad + \left( 2HTr(\mathcal{H}^3) -\frac{5}{2}H^2|\mathcal{H}|^2 + \frac{1}{2}H^4 \right) gu \;do
\end{align*}
which unveils the symmetry of the second variation.
\end{remark}

\begin{corollary}
\label{Corollary_2var_Willmore_on_sphere}
If $\Gamma = S(0,R)$, an $n-$sphere, then the second variation of the Willmore functional $W(\Gamma)$ is given by
\begin{align*}
& W''(\Gamma)[u\nu, g\nu] = \int_\Gamma \DeltaG g \DeltaG u + \frac{n}{R^2}\left( \frac{3n}{2} - 4 \right) \nablaG g \cdot \nablaG u +\frac{n^2}{R^4}\left( \frac{n^2}{2} - \frac{5n}{2} +3  \right)gu \; do.
\end{align*}
\end{corollary}
\begin{proof}
Inserting $\mathcal{H} = \frac{1}{R} P$ and $H = \frac{n}{R}$ into the second variation
and using integration by parts produces the result.
\end{proof}

\subsection*{Derivation of the second variation of the area and volume functionals}
We also require the first and second variation of the area and volume functionals, which are
$$
	A(\Gamma) := \int_\Gamma 1 \; do \quad \text{ and } \quad V(\Gamma) = \frac{1}{n+1}\int_\Gamma id_\Gamma \cdot \nu \;do.
$$
\begin{corollary}
\label{Corollary_2var_area_functional}
The first and second variation of the area functional $A(\Gamma)$ are given by 
\begin{align*}
	&	A'(\Gamma)[u\nu] = \int_\Gamma u H \; do, \\
	&	A''(\Gamma)[u\nu, g\nu] = \int_\Gamma u g H^2  - u (\Delta_{\Gamma} g + |\mathcal{H}|^2 g) \; do.
\end{align*}
\end{corollary}	
\begin{proof}
The transport formula \eqref{transport_formula} directly gives the first variation. The second variation is then obtained from Corollary \ref{mean_curvature_1var}.
\end{proof}
Using integration by parts, we obtain
$$
	A''(\Gamma)[u\nu, g\nu] = \int_\Gamma  \nablaG u \cdot \nablaG g + (H^2 - |\mathcal{H}|^2) u g  \; do.
$$

\begin{corollary}
\label{Corollary_2var_volume_functional}
The first and second variation of the volume functional $V(\Gamma)$ are given by
\begin{align*}
	&	V'(\Gamma)[u\nu] = \frac{1}{n+1}\int_\Gamma u -id_\Gamma \cdot \nablaG u + id_\Gamma \cdot \nu u H \; do, \\
	&	V''(\Gamma)[u\nu, g\nu] = \frac{1}{n+1}\int_\Gamma g \A \nablaG u \cdot id_{\Gamma} - id_{\Gamma} \cdot \nu (\nablaG u \cdot \nablaG g + u\DeltaG g ) -H id_{\Gamma} \cdot( u\nablaG g + g \nablaG u ) \\
& \qquad +\left( 2H -  id_{\Gamma}\cdot\nu|\A|^2 +H^2 id_{\Gamma} \cdot\nu \right)gu \; do.
\end{align*}
\end{corollary}	
\begin{proof}
The transport formula \eqref{transport_formula} directly gives the first variation. The second variation is then obtained from Corollary \ref{mean_curvature_1var} and Lemma \ref{materialderiv_deriv}.
\end{proof}
Using integration by parts, we obtain
\begin{align*}
& V'(\Gamma)[u\nu] = \int_\Gamma u \;do, \label{volume1var}\\
& V''(\Gamma)[u\nu,g\nu] = \int_\Gamma Hgu \;do.
\end{align*}

\newpage
\bibliography{refs}{}
\bibliographystyle{plain}

\end{document}